\documentclass[11pt,a4paper]{article}

\usepackage[a4paper,margin=1in]{geometry}
\usepackage[utf8]{inputenc}
\usepackage{subfiles}
\usepackage{mathtools}
\usepackage[thinc]{esdiff}
\usepackage{array}
\usepackage{gensymb}
\usepackage{amssymb}
\usepackage{amsmath, bm}
\usepackage{amsthm}
\usepackage{amsfonts}
\usepackage{braket}
\usepackage[super]{nth}
\usepackage[sorting=none]{biblatex}
\usepackage{caption}
\usepackage{subcaption}
\allowdisplaybreaks
\usepackage{placeins}
\usepackage{float}
\usepackage{tikz}
\usepackage{breqn}
\usepackage{hyperref}
\usepackage{bm}
\usetikzlibrary{shapes.geometric, arrows}
\usetikzlibrary{trees}
\usepackage[toc,page]{appendix}
\usetikzlibrary{quantikz2}
\usetikzlibrary{calc}
\usepackage{tikz-cd}
\usepackage{yfonts}
\usepackage{paralist}
\usepackage{circuitikz}
\usetikzlibrary{shapes.geometric}
\usepackage{relsize}
\usetikzlibrary{decorations.markings}
\usepackage{algorithm}
\usepackage{algpseudocode}
\usepackage[utf8]{inputenc}
\usepackage{physics}
\usepackage{nicematrix}
\usepackage{rotating}
\usepackage{booktabs}
\usepackage{subcaption}   
\usepackage{siunitx}
\usepackage{comment}
\usepackage{booktabs}

\newcommand{\R}{\mathbb R}

\theoremstyle{definition}
\newtheorem{theorem}{Theorem}[section]

\newtheorem{lemma}[theorem]{Lemma}
\newtheorem{proposition}[theorem]{Proposition}
\newtheorem*{remark}{Remark}

\title{Composition and tensor train structure in polynomial optimization}
\author{Lloren\c{c} Balada Gaggioli \thanks{
LAAS - CNRS, Université de Toulouse, France. Emails: llorenc.balada.gaggioli@fel.cvut.cz, henrion@laas.fr, korda@laas.fr}\textsuperscript{\;\;,}\thanks{
Faculty of Electrical Engineering, Czech Technical University in Prague, Czechia}  
\and Didier Henrion\footnotemark[1]\textsuperscript{\;\;,}\footnotemark[2]
\and Milan Korda\footnotemark[1]\textsuperscript{\;\;,}\footnotemark[2]}

\date{\today}

\usepackage{xcolor}
\definecolor{dkblue}{rgb}{0,0,0.8}

\begin{document}

\maketitle

\begin{abstract}
    
    We study polynomial optimization problems whose objective has a composition or tensor train structure. These polynomials can be evaluated as a sequence of maps, giving rise to intermediate variables (``states'') of dimension lower than the ambient dimension. Structures like these arise naturally in dynamical systems, Markov chains, and neural networks. We develop two moment-SOS (sums of squares) hierarchies that exploit this composition structure in different ways. The first one, termed state-lifting chordal, is based on the correlative sparsity graph of the problem. The second one, termed state-lifting push-forward, encodes the structure at the level of the measures directly. Numerical experiments demonstrate that the proposed methods can compute certified bounds for problems with hundreds or even a thousand variables. To illustrate the versatility of the hierarchies we apply them to Markov chain optimization, quantum optimal control, and neural networks.
\end{abstract}

\section{Introduction}

The \emph{moment-SOS (sums of squares) hierarchy} allows one to compute guaranteed bounds and sometimes the global optimum
of a multivariate polynomial subject to polynomial constraints; see \cite{HenrionKordaLasserre2020} for an overview
and applications. These bounds are computed by solving convex (typically
semidefinite) optimization problems of increasing size. For polynomial optimization problems without specific structure,
the size of these convex problems grows rapidly, restricting this approach to small dimensions only.
This motivated the study of structure exploiting approaches to overcome the curse of dimensionality.
Based on graph theoretical notions (chordal graph completion, clique tree decomposition), correlative and term
sparsity structures can be leveraged to improve significantly the scalability of the moment-SOS hierarchy.
The computational cost is no longer driven by the number of variables, but rather by the size of the largest
clique in the graph describing the correlations between the variables.
Tight bounds on the global optima of polynomials with thousands of variables and constraints have been obtained; see \cite{MagronWang23} for an overview.

This work proposes a new type of structure amenable for exploitation, which we refer to as the composition structure. Here, a polynomial in $n$ variables $p(x_1,\ldots,x_n)$ is written as
\[
p(x_1,\ldots,x_n) = F_n(\cdot,x_n)\circ F_{n-1}(\cdot,x_{n-1})\circ\ldots\circ F_2(\cdot,x_2)\circ F_1(x_1) ,
\] 
where the composition is with respect to the first arguments of the polynomial mappings $F_i: \R^{r_{i-1}}\times \R \to \R^{r_i}$ ($i = 2,\ldots,n$, $r_n=1$) and $F_1:\R\to\R^{r_1}$. Provided that the ranks $r_i$ are small compared to the ambient dimension $n$, this structure allows for a significant speedup when minimizing the polynomial using the moment-sum-of-squares hierarchies.  This is achieved by introducing auxiliary \emph{lifting variables}, $s_1,\ldots,s_n$, by setting 
\[
s_1=F_1(x_1),\quad s_i=F_i(s_{i-1},x_i)\, \text{ for }\, i=2,\dots,n, \quad p(x_1,\ldots,x_n) = s_n.
\]
In order to exploit the structure, we propose two methods. The first one uses correlative sparsity~\cite{Waki2006,Lasserre2006}, whereas the second one uses the so-called push-forward operators commonly used in the study of discrete-time dynamical systems or Markov processes (e.g., \cite{hernandez2012discrete}). The former leads to groups of variables indexing the moment matrices of maximal size $ \max_{i}(r_i + r_{i+1})+1$. The latter leads to groups of variables of maximal size  $ \max_{i}r_i  + 1$. The price to pay for the smaller group size of the push-forward approach is higher degree of the polynomials involved and we analyze the tradeoff both theoretically and numerically.

The composition structure is a very general modelling framework, with notable applications involving dynamical systems, neural networks or Markov processes. For dynamical systems, the additional lifting variables represent the state variable of the system and all ranks are equal to the state-space dimension. For neural networks, the lifting variables are the outputs of the hidden layers of the network. In all these examples, the polynomial in the original variables may be fully dense, while its computation remains local, step by step, through intermediate quantities of controlled dimension. This is precisely the type of low-dimensional structure that the proposed hierarchies are designed to exploit.

A particularly important special case is given by Tensor Train (TT) polynomials, where each mapping $F_i$ is linear in its first coordinate; this is equivalent to the existence of a TT decomposition~\cite{oseledets2011tensor} of the polynomial’s coefficient tensor with ranks equal to $r_i$. In other words a tensor-train polynomial $p$ can be written as \begin{equation}\label{ttpol}
p(x_1,\ldots,x_n) = \prod_{i=1}^n P_i(x_i),
\end{equation}
where $P_i$ are polynomial matrices of size $r_{i-1}\times r_{i}$. This perspective also clarifies the relation with our previous work on low-rank polynomial optimization \cite{LRPOP25}, where we consider polynomials in the Canonical-Polyadic (CP) format
\[
p(x_1,\ldots,x_n) = \sum_{l=1}^r\prod_{i=1}^n p_{l,i}(x_i),
\]
that is sums of products of univariate polynomials. By introducing suitable lifting variables, we showed that the CP structure induces a sparse lifted formulation of the moment-SOS hierarchy whose complexity depends on the rank $r$ rather than the ambient dimension. It is important to note that TT generalizes CP, in the sense that CP structure can be recovered from a TT structure where the cores are diagonal and the intermediate ranks are equal to the CP rank. 

From the tensor perspective, TT is also a natural alternative to CP, since it is generally better behaved as an approximation format: unlike CP, best low-rank TT approximations always exist, and stable constructive procedures are available \cite{oseledets2011tensor,deSilvaLim08}. At the same time, this work shows that the essential principle is broader than tensor multilinearity itself, what matters is the propagation of low-dimensional information through a chain of local maps, which naturally leads to the nonlinear composition framework studied here. Therefore, the fundamental idea is that certain high-dimensional polynomials can sometimes be evaluated via the propagation of low-dimensional information, and this structure can be exploited in polynomial optimization.

\paragraph{Contribution}
The main contribution of this work is to identify the composition structure as a new form of exploitable structure in polynomial optimization, through the introduction of intermediate lifting variables. This allows for solving  potentially very high-dimensional problems (i.e. large $n$) provided that the ranks $r_i$ are small. We rigorously prove convergence of the hierarchies exploiting this structure for both methods of exploitation: correlative sparsity and push-forward operators. We analyze the computational complexity of the method and demonstrate its numerical performance on a broad set of examples spanning quantum optimal control, Markov processes, and neural networks. Correlative sparsity exploitation for problems exhibiting chain-like structure has been proposed before in~\cite{teng2025convex,kang2024fast}. The key novelty of our work is that this structure can be \emph{created} by introducing additional lifting variables even though the original problem is fully dense. For certain classes of problems such as tensor train polynomials, this structure can even be detected directly from the coefficient tensor of the polynomial using tailored algorithms~\cite{oseledets2011tensor}. The structure exploitation based on push-forward operators seems new in the context of sparse polynomial optimization, even though it is a natural tool to model this problem class on the infinite-dimensional level of occupation measures, heavily used in discrete-time Markov processes literature (e.g., \cite{hernandez2012discrete}) as well as more recently in discrete-time dynamical systems (e.g., \cite{korda2014convex}).

\paragraph{Outline}
In Section \ref{sec:ttcomp} we describe the composition structure  that we  address in our new framework. We show that it includes the classical multilinear TT structure \eqref{ttpol} as a particular case. Section \ref{sec:momsos} is a self-contained description of the moment-SOS hierarchy for polynomial optimization problems with correlative sparsity structure. In particular, in Subsection \ref{sec:lrpop} we recall how our previous framework can exploit the low rank canonical polyadic tensor structure. The chordal and push-forward moment-SOS hierarchies are described in Section \ref{sec:chordal} and \ref{sec:push-forward}, with proofs of convergence, and a detailed complexity analysis. Both hierarchies are then illustrated and compared on a broad collection of numerical examples in Section \ref{sec:examples}. Finally, a varied set of applications of the methods presented are detailed in Section \ref{sec:applications}, including Markov chains, quantum optimal control and feedforward neural networks.

\paragraph{Notation}

We write $p(x)$ for a scalar polynomial, $F_i(s_{i-1},x_i)$ for maps from $\R^{r_{i-1}}\times\R$ to $\R^{r_{i}}$, and $P(x)$ for matrices with polynomial entries. The dimensions of the states are denoted $r_i$, $i=0,1,\ldots,n$, and referred to also as ranks, or bond dimensions in the context of tensor trains. For tensor trains, we call each of the matrices a tensor core.

\section{Compositional polynomials}\label{sec:ttcomp}

Many dense polynomials have an underlying structure that is not visible in their monomial expansion. In this work, we study polynomials that present a compositional structure, and the linear case of these, which are tensor train polynomials. This compositional structure can be exploited in two different ways, each coming with its advantages and drawbacks. The first way is based on the correlative sparsity graph arising from the lifting of the state variables, which we will call state-lifting chordal hierarchy. The other one is based on push-forward measure propagation, and we will call it state-lifting push-forward hierarchy. 

Let $x=\{x_1,\dots,x_n\}$, where $x_i$ can be a vector in general but we will use scalar $x_i \in \R$ for simplicity. We say that a polynomial $p(x)\in \R[x]$ presents a chain composition structure if there exist integers $r_0=1,r_1,\dots,r_{n-1},r_n=1$ and polynomial maps
\[
F_i:\R^{r_{i-1}}\times \R \rightarrow \R^{r_i},
\]
such that the polynomial $p(x)$ can be computed in the following way
\begin{equation}\label{eq:themodel}
    s_1 =F_1(x_1), \quad
    s_i  = F_i(s_{i-1},x_i), \quad
    p(x)=s_n,
\end{equation}
where $s_i\in \R^{r_i}$ are intermediate state variables. With this we have represented $p(x)$ as a one dimensional chain where, at each step, only the previous state and the local variable interact. The values of $r_i$ encode the amount of information propagated at each step. We denote the aggregate of lifting variables $s := (s_1,\ldots,s_n)$.

\paragraph{Tensor train polynomials} Now let us consider the specific case where at each step we have bilinearity in the state. We will call these tensor train polynomials. We start from a tensor perspective and see that it results exactly in the bilinear case of the compositional polynomials. 

A polynomial $p(x)$ can be written in terms of its coefficient tensor, $C$, for some univariate basis $\{\phi_{i,\alpha}\}_{\alpha\in \mathcal{A}_i}$ (e.g. monomials $\phi_{i,\alpha}(x_i)=x_i^{\alpha}$ up to degree $d_i$) as
\begin{equation}
  p(x)=\sum_{\alpha_1\in\mathcal A_1}\cdots\sum_{\alpha_n\in\mathcal A_n}
  C_{\alpha_1,\dots,\alpha_n}\ \prod_{i=1}^n \phi_{i,\alpha_i}(x_i),
\end{equation}

with $a_i=|\mathcal{A}_i|$, where $\mathcal{A}_i$ are index sets. Depending on the characteristics of the coefficient tensor, we can decompose it into a simpler form, such that we have to store less information. Here we study the train decomposition \cite{oseledets2011tensor,dolgov2013tensor}, also known as matrix product state in quantum physics. The coefficient tensor can be written as 
\begin{equation}
    C_{\alpha_1,\dots,\alpha_n}= U^{(1)}_{\alpha_1}U^{(2)}_{\alpha_2}\cdots U^{(n)}_{\alpha_n}, 
\end{equation}
with $U^{(i)}_{\alpha_i}\in \R^{r_{i-1},r_{i}}$ for $i=1,\dots,n$, with $r_0=r_n=1$, as before. We have a chain of tensor cores of dimension 3, which allows us to define the matrices of polynomials
\begin{equation}\label{eq:Pi}
    P_i(x_i) = \sum_{\alpha_i=1}^{a_i}U^{(i)}_{\alpha_i}\phi_{i,\alpha_i}(x_i) \in \R^{r_{i-1}\times r_i}.
\end{equation}

With these we can now write the polynomial as
\begin{equation}
    p(x)=\prod_{i=1}^n \Bigg(\sum_{\alpha_i=1}^{a_i}U^{(i)}_{\alpha_i}\phi_{i,\alpha_i}(x_i)\Bigg) =  P_1(x_1)P_2(x_2)\cdots P_n(x_n).
\end{equation}
We can see this as a product of univariate polynomial matrices. Now, the composition will look like
\begin{align*}
    s_1 &=P_1(x_1) \\
    s_i & = F_i(s_{i-1},x_i) := s_{i-1}P_i(x_i) \\
    p(x)&=s_n,
\end{align*}
where each step is bilinear in $(s_{i-1},x_i)$, and we call the integers $r_i$ the tensor ranks, or bond dimensions. Note that the labeling of the variables can be turned around such that we can write $s_i = F_i(s_{i-1},x_i) := P_i(x_i)s_{i-1}$ instead.

\subsection{Examples}

\subsubsection{A nonlinear composition map}

Consider the polynomial defined recursively by
\[
s_1=x_1,
\qquad
s_i=s_{i-1}^2+x_i,
\qquad i=2,\dots,n,
\]
and let
\[
p_n(x_1,\dots,x_n)=s_n.
\]
This is a chain composition map in which each state variable \(s_i\) depends only on the previous state \(s_{i-1}\) and the local variable \(x_i\).

The resulting polynomial is neither separable nor sparse. Indeed,
\[
p_2(x_1,x_2)=x_1^2+x_2,
\]
\[
p_3(x_1,x_2,x_3)=(x_1^2+x_2)^2+x_3
= x_1^4+2x_1^2x_2+x_2^2+x_3,
\]
and further iterations generate many mixed monomials involving variables from different stages. Thus the correlative sparsity pattern rapidly becomes dense.

\subsubsection{A tensor train example}

For the polynomial
\[
p_n(x_1,\dots,x_n)=\prod_{i=1}^{n-1}(1+x_i x_{i+1}),
\]
there is a natural tensor train representation in the ordering \((x_1,\dots,x_n)\):
\[
p_n(x_1,\dots,x_n)
=
\begin{bmatrix}1 & x_1\end{bmatrix}
\prod_{i=2}^{n-1}
\begin{bmatrix}
1 & x_i\\
x_i & x_i^2
\end{bmatrix}
\begin{bmatrix}
1\\
x_n
\end{bmatrix}.
\]
Therefore all intermediate TT ranks are equal to \(2\), namely
\[
(r_1,\dots,r_{n-1})=(2,\dots,2).
\]
So the maximal TT rank is uniformly bounded, independently of \(n\). This is a genuine TT example: the information propagates linearly through the chain, and the polynomial is exactly represented by small TT cores.

\subsubsection{Impact of the ordering of the variables}

For the polynomial of the previous section, the maximal TT rank is \(2\), independently of \(n\). This reflects the fact that, in the natural ordering, each bipartition \(\{x_1,\dots,x_i\}\mid\{x_{i+1},\dots,x_n\}\) cuts only one edge of the path graph, so a two-dimensional space is sufficient to transmit the local state across the chain. By contrast, if one chooses a nonlocal ordering, for instance an odd-even ordering
\[
(x_1,x_3,x_5,\dots,x_2,x_4,x_6,\dots),
\]
then many nearest-neighbor interactions \(x_i x_{i+1}\) are split simultaneously across the middle bonds, and the TT ranks can increase dramatically. Indeed, TT ranks are not intrinsic to the polynomial alone, they depend crucially on the chosen ordering of the variables. For example when $n=4$, using the odd-even order \((x_1,x_3,x_2,x_4)\), 
one obtains
\[
p_4(x_1,x_2,x_3,x_4)=
\begin{bmatrix}1&x_1\end{bmatrix}
\begin{bmatrix}
1 & 0 & x_3 & 0 & x_3^2 & 0\\
0 & 1 & 0 & x_3 & 0 & x_3^2
\end{bmatrix}
\begin{bmatrix}
1 & 0\\
x_2 & 0\\
x_2 & 1\\
x_2^2 & x_2\\
0 & x_2\\
0 & x_2^2
\end{bmatrix}
\begin{bmatrix}
1\\
x_4
\end{bmatrix},
\]
whose TT rank profile is
\[
(r_1,r_2,r_3)=(2,6,2).
\]
For \(n=6\), the odd-even TT rank profile is
\[
(r_1,r_2,r_3,r_4,r_5)=(2,6,18,6,2).
\]

\subsubsection{From a nonlinear chain composition to a tensor train}

Since any polynomial can be described in a tensor train format, a natural question is whether one can focus solely on the  tensor train polynomials, i.e. those for which the recurrence relation is linear in the $s$ variables. As it turns out, this can be highly suboptimal as is demonstrated by the following example. Consider as before the polynomial defined recursively by
\[
s_1=x_1,
\qquad
s_i=s_{i-1}^2+x_i,
\qquad i=2,\dots,n,
\]
and let
\[
p_n(x_1,\dots,x_n)=s_n.
\]

We now discuss how such a nonlinear composition may be converted into a tensor train representation after lifting. In order to linearize the update
\[
s_i=s_{i-1}^2+x_i,
\]
one has to enlarge the state so as to include suitable monomials in \(s_{i-1}\). For instance, if one introduces the lifted state
\[
\bigl[\,1,\ s_{i-1},\ s_{i-1}^2\,\bigr],
\]
then the update of the first two components is given by 
\[
1 \mapsto 1,
\qquad
s_i=s_{i-1}^2+x_i.
\]
However, in order to propagate the lifted state further, one must also represent
\[
s_i^2=(s_{i-1}^2+x_i)^2=s_{i-1}^4+2x_i s_{i-1}^2+x_i^2,
\]
which requires the additional monomial \(s_{i-1}^4\). At the next step, one similarly needs \(s_{i-1}^8\), and so on. Thus an exact linearization requires a hierarchy of lifted monomials
\[
1,\ s_i,\ s_i^2,\ s_i^4,\ \dots
\]
whose size grows with the length of the chain.

Therefore, this nonlinear composition does not admit an exact tensor train representation with uniformly bounded lifted state dimension. Passing from the nonlinear chain composition to a tensor train requires a growing lifted state, and hence increasing TT ranks. This illustrates why, in the paper, tensor train polynomials are restricted to the case where the propagated state has the form
\[
s_i=(s_{i,1},\dots,s_{i,r_i})^\top,
\]
with updates that are linear in the state variables after fixing \(x_i\). The present example falls naturally within the broader class of nonlinear composition maps, but not within the class of tensor train polynomials with uniformly bounded rank.

\section{Moment-SOS hierarchies}\label{sec:momsos}

Polynomial optimization can be studied through two complementary
representations: moment sequences and sum-of-squares (SOS) polynomials.
We briefly recall these notions before introducing the standard dense
semidefinite relaxations \cite{Lasserre2001} and its sparse counterpart. After that we turn to exploiting the new types of sparsity introduced in this work.

\noindent\textbf{Moments.}
Let $x=(x_1,\dots,x_n)$ and consider the monomial basis
$\{x^{\bm{\alpha}}\}_{\bm{\alpha}\in\mathbb{N}^n}$ of
$\mathbb{R}[x]$.  
We associate with this basis a sequence
$\mathbf{y}=(y_{\bm{\alpha}})$ indexed by the same multi-indices.
Using this sequence, we define a linear functional
$L_{\mathbf{y}}:\mathbb{R}[x]\rightarrow\mathbb{R}$ acting on
polynomials as
\begin{equation}
p(x)=\sum_{\bm{\alpha}} p_{\bm{\alpha}}x^{\bm{\alpha}}
\quad\mapsto\quad
L_{\mathbf{y}}(p)=\sum_{\bm{\alpha}} p_{\bm{\alpha}}y_{\bm{\alpha}} .
\end{equation}

For a relaxation order $k$, the \emph{moment matrix} $M_k(\mathbf{y})$
is indexed by monomials of degree at most $k$ and defined by
\begin{equation}
M_k(\mathbf{y})_{\bm{\beta},\bm{\gamma}}
=
L_{\mathbf{y}}(x^{\bm{\beta}}x^{\bm{\gamma}})
=
y_{\bm{\beta}+\bm{\gamma}} .
\end{equation}

Given a polynomial
$g(x)=\sum_{\bm{\alpha}} g_{\bm{\alpha}}x^{\bm{\alpha}}$,
one defines the associated \emph{localizing matrix}
\begin{equation}
M_k(g\mathbf{y})_{\bm{\beta},\bm{\gamma}}
=
L_{\mathbf{y}}(g\,x^{\bm{\beta}}x^{\bm{\gamma}})
=
\sum_{\bm{\alpha}} g_{\bm{\alpha}}
y_{\bm{\alpha}+\bm{\beta}+\bm{\gamma}} .
\end{equation}

These matrices encode positivity conditions satisfied by moment sequences
associated with measures supported on semialgebraic sets.

\medskip
\noindent\textbf{Sums of squares.}
Let $p(x)\in\mathbb{R}[x]$ be a polynomial.
We say that $p$ is a \emph{sum of squares} (SOS) if it can be written as
\begin{equation}
p(x)=\sum_{i=1}^L q_i(x)^2
\end{equation}
for some polynomials $q_1,\dots,q_L\in\mathbb{R}[x]$.
Any polynomial admitting such a decomposition is necessarily
nonnegative.

Suppose that $p$ has degree $2d$ and let $\mathbf{z}_d(x)$
denote the vector of all monomials in $x$ of degree at most $d$.
Then $p$ is SOS if and only if there exists a positive semidefinite
matrix $Q\succeq0$ such that
\begin{equation}
p(x)=\mathbf{z}_d(x)^{\!\top}Q\mathbf{z}_d(x).
\end{equation}

The matrix $Q$ is called a \emph{Gram matrix}. This characterization
connects polynomial nonnegativity certificates with semidefinite
programming.

\subsection{Dense hierarchy}

Consider the polynomial optimization problem
\begin{equation}
\begin{aligned}
\min_{x\in\mathbb{R}^n} \quad & p(x) \\
\text{s.t.}\quad & g_j(x)\ge0, \qquad j=1,\dots,m .
\label{eq:pop}
\end{aligned}
\end{equation}

Let
\[
K=\{x\in\mathbb{R}^n \mid g_j(x)\ge0,
\; j=1,\dots,m\}
\]
denote the feasible set.
The problem can be written as
\begin{equation}
p^*=\inf_{x\in K} p(x).
\end{equation}

An equivalent dual viewpoint is
\begin{equation}
\lambda^*=\sup\{\lambda \mid p(x)-\lambda \ge0,
\;\forall x\in K\}.
\end{equation}

Introducing $g_0=1$, the moment relaxation of order $k$ reads
\begin{equation}
\begin{aligned}
(P_k)\qquad
p_k=\inf_{\mathbf{y}} \quad & L_{\mathbf{y}}(p) \\
\text{s.t.}\quad
& L_{\mathbf{y}}(1)=1, \\
& M_{k-d_j}(g_j\mathbf{y})\succeq0, \qquad j=0,\dots,m ,
\end{aligned}
\end{equation}
where $d_j=\lceil\deg(g_j)/2\rceil$.

The dual SDP corresponds to the SOS program
\begin{equation}
\begin{aligned}
(D_k)\qquad
\lambda_k=\sup_{\lambda,\sigma_j} \quad & \lambda \\
\text{s.t.}\quad
& p(x)-\lambda
=
\sum_{j=0}^m \sigma_j(x)g_j(x), \\
& \sigma_j \text{ SOS}, \quad j=0,\dots,m .
\end{aligned}
\end{equation}

Equivalently,
\[
p(x)-\lambda \in \mathcal{Q}(\mathbf{g}),  \text{ where } \quad
\mathcal{Q}(\mathbf{g})
=
\left\{
\sum_{j=0}^m \sigma_j(x)g_j(x)
:\sigma_j \text{ SOS}
\right\}
\]
is the quadratic module generated by $\mathbf{g}=(g_0,\ldots, g_m)$.

If the quadratic module is \emph{Archimedean}, i.e., there exists
$R>0$ such that $R^2-\|x\|^2\in\mathcal{Q}(\mathbf{g})$,
then Putinar's Positivstellensatz guarantees convergence of the
hierarchy \cite{Lasserre2001}:
\begin{equation}
p_k \nearrow p^*, \qquad \lambda_k \nearrow p^* .
\end{equation}

This construction is referred to as the \emph{dense moment-SOS hierarchy}
because all variables appear jointly in the moment and Gram matrices,
leading to semidefinite blocks of size
\[
\binom{n+k}{k}.
\]
When the polynomial problem exhibits additional structure or sparsity,
specialized hierarchies can exploit this structure to produce
smaller semidefinite programs.

\subsection{Sparse hierarchies}

The dense hierarchy couples all variables together in a single moment
matrix. However, many polynomial optimization problems exhibit a
structured interaction pattern between variables. This structure can be
exploited through the notion of \emph{correlative sparsity} \cite{fukuda2001exploiting,kim2009generalized}.

Given the problem \eqref{eq:pop}, we construct the
\emph{correlative sparsity graph} $G=(V,E)$ \cite{Waki2006}, where the vertex set
$V=\{1,\dots,n\}$ represents the variables.
An edge $(i,i')\in E$ is introduced whenever the variables $x_i$ and
$x_{i'}$ appear together in a monomial of the objective polynomial $p$
or in one of the constraint polynomials $g_j$.
This graph therefore captures which variables are directly coupled in
the polynomial data.

To exploit this structure, we recall several standard notions from graph
theory.
A graph is said to be \emph{chordal} if every cycle containing four or
more vertices has a chord, that is, an edge connecting two nonconsecutive
vertices of the cycle \cite{Vandenberghe2015}.
A \emph{clique} is a subset of vertices such that every pair of vertices
in the subset is connected by an edge.

If the sparsity graph is not chordal, one can construct a
\emph{chordal extension} $\widetilde G$ by adding edges until the graph
becomes chordal.
The resulting graph admits a decomposition into cliques
$\mathcal{I}=\{I_1,\dots,I_N\}$ that satisfy the
\emph{running intersection property} (RIP): whenever two cliques share a
variable, all cliques on the path connecting them also contain that
variable.
When the cliques are maximal, this decomposition forms a tree
structure, commonly called a \emph{clique tree}.
Throughout this work we assume that the cliques
$I_1,\dots,I_N$ correspond to the maximal cliques of such a chordal
extension.

Each clique $I_a$ contains a subset of variables $x_{I_a}$, and the
polynomial data can be distributed among these cliques.
In particular, the objective polynomial can be written as
\[
p(x)=\sum_{a=1}^N p_a(x_{I_a}),
\]
where each $p_a$ depends only on the variables in clique $I_a$.
Similarly, each constraint polynomial $g_j$ is assigned to a clique
containing all of its variables; we denote by
$J_a\subset\{1,\dots,m\}$ the set of constraints associated with clique
$I_a$.

The sparse moment relaxation then introduces a separate moment sequence
for each clique.
The associated moment and localizing matrices are constructed only in
the variables $x_{I_a}$.
Consistency between different cliques is enforced through equality
constraints on the moments corresponding to shared variables.
These overlap constraints are defined on the \emph{separators}
$S_{ab}=I_a\cap I_b$ of the clique tree.

The key advantage of this construction is that the size of the
semidefinite blocks depends only on the number of variables inside each
clique rather than on the total number of variables $n$.
More precisely, the moment and Gram matrices associated with clique
$I_a$ have dimension
\[
\binom{|I_a|+k}{k},
\]
which can be significantly smaller than the dense size
$\binom{n+k}{k}$ when the cliques are small.
The overall SDP therefore consists of several smaller semidefinite
blocks coupled by linear consistency constraints.

Under a suitable \emph{sparse Archimedean condition} \cite{Lasserre2006}, namely the
existence of a decomposition
\[
R-\|x\|_2^2 = \sum_{a=1}^N q_a(x_{I_a})
\]
with each $q_a$ belonging to the quadratic module generated by the
constraints assigned to clique $I_a$, the resulting hierarchy is
monotone and convergent:
\[
p_k^{\mathrm{cs}}\nearrow p^*, \qquad
\lambda_k^{\mathrm{cs}}\nearrow p^*,
\]
as $k\to\infty$ \cite{Lasserre2006,Waki2006,Grimm2007}.
This approach is commonly referred to as the
\emph{correlative sparse SOS} (CSSOS) hierarchy. The convergence rate of this hierarchy is understood and indeed one observes an asymptotic speed-up compared to the dense hierarchy provided that the maximum clique size is sufficiently small compared to the ambient dimension~\cite{korda2025convergence}. 

Another type of sparsity arises when the polynomial contains only a
small number of monomials.
In this case, one can construct a graph whose vertices correspond to the
terms appearing in the problem data and apply a similar chordal
decomposition procedure.
This leads to the \emph{term sparsity SOS} (TSSOS) hierarchy
\cite{Wang2021}.
While CSSOS exploits sparsity in the interaction between variables,
TSSOS instead exploits sparsity in the set of monomials. Hybrid strategies combining both ideas have also been proposed
\cite{Wang2022}, together with other variants that exploit additional
structure such as ideal sparsity \cite{Korda2024} or a hidden low-dimensional structure~\cite{lasserre2022optimization}.

\subsection{Rank-sparse hierarchy}\label{sec:lrpop}

Another type of structure that can be exploited in polynomial
optimization arises when the objective admits a low-rank tensor
decomposition \cite{LRPOP25}. A polynomial $p(x)$ in $n$ variables is said to
have rank $r$ if it can be written as
\begin{equation}\label{eq:cp_poly}
p(x) = \sum_{l=1}^r \prod_{i=1}^n p_{l,i}(x_i),
\end{equation}
where each $p_{l,i}$ is a univariate polynomial.
This representation corresponds to a Canonical Polyadic (CP)
decomposition of the coefficient tensor of $p$.

To exploit this structure, additional variables are introduced that
recursively accumulate the partial products. For each rank component
$l=1,\dots,r$, we define lifting variables
\[
s_{l,1}=p_{l,1}(x_1), \qquad
s_{l,i}=s_{l,i-1}p_{l,i}(x_i), \quad i=2,\dots,n .
\]
These relations can be written as polynomial equality constraints
\[
h_{l,1}(t,x)=s_{l,1}-p_{l,1}(x_1), \qquad
h_{l,i}(t,x)=s_{l,i}-s_{l,i-1}p_{l,i}(x_i),
\]
so that the objective becomes
\[
p(x)=\sum_{l=1}^r s_{l,n}.
\]

The polynomial optimization problem can therefore be reformulated in the
extended variables $(x,s)$ with equality constraints
encoding the recursive relations above.
The resulting correlative sparsity graph has a structured form,
denoted $G_{r,n}$, whose vertices correspond to the variables
$\{x_i,s_{l,i}\}$.
This graph depends only on the rank $r$ and the number of variables $n$.

A key property of this construction is that the maximal clique size of a
chordal extension of $G_{r,n}$ is bounded independently of $n$.
In particular, the treewidth of the graph satisfies
\[
\mathrm{tw}(G_{r,n}) = \mathcal{O}(r),
\]
so the largest cliques contain at most $\mathcal{O}(r)$ variables.
As a consequence, the moment and Gram matrices arising in the
corresponding sparse moment-SOS hierarchy have size
\[
\binom{r+k}{k},
\]
which depends only on the rank and the relaxation order, but not on the
number of variables.

This hierarchy, referred to as \emph{low-rank polynomial optimization}
(LRPOP), therefore scales linearly with the number of variables while
the main computational cost is governed by the rank of the polynomial.
This property allows global optimization of problems with a very large
number of variables provided the rank remains small.

LRPOP is the precursor of the framework we are currently studying: composition of polynomials. It exploits the fact that we can encode the low-dimensional transfer of information by introducing intermediate quantities, and the rank encodes the complexity of the global transfer of information. Similarly, in our composition framework the ranks (or state dimensions) encode how complex is the transfer of information for each time step. If we consider a TT with only diagonal elements in its cores, and bond dimensions $r$ for all cores, then we recover the form of the CP decomposition. Therefore, a TT can generalize a CP decomposition in this sense.

\section{State-lifting chordal hierarchy}\label{sec:chordal}

In this first hierarchy we propose, we exploit the compositional structure~\eqref{eq:themodel} by leveraging the correlative sparsity it induces after introducing state-lifting variables, and studying the properties of the graph it generates. The block sizes in the resulting hierarchy will depend on the state dimensions (ranks) instead of the original number of variables. We want to solve the optimization problem
\begin{equation}
\begin{aligned}
\min_{x\in\mathbb{R}^n} \quad & p(x) \\
\text{s.t.}\quad & g_i(x_i)\ge0, \qquad i=1,\dots,n,
\label{eq:pop_sparse}
\end{aligned}
\end{equation}
where we assume that the set  $\{x \mid g_i(x_i) \ge 0\}$ is compact. The polynomial $p$ is given by the previously defined composition structure~\eqref{eq:themodel}

\[
s_{1} =F_1(x_1), \quad  s_i  = F_i(s_{i-1},x_i), \quad p(x)=s_n,
\]
where $F_1:\R^1\mapsto \R^{r_1}$, $F_i:\R^{r_{i-1}+1}\mapsto \R^{r_i}$ are polynomial maps, $s_i\in \R^{r_i}$ and $r_0=r_n=1$. In order to capture the  constraints defining the composition structure, we introduce the polynomials \[
h_{1,l}(s_{1},x_1)=s_{1,l}-F_{1,l}(x_1)=0, \qquad l=1,\dots,r_1
\]
and, for each $i = 2,\ldots,n$,
\[
h_{i,l}(s_{i-1},s_{i},x_i)=s_{i,l}-F_{i,l}(s_{i-1},x_i), \; l=1,\ldots,r_i
\] In these constraints, the variable coupling becomes evident, with the cliques being 
\[
\{s_{i-1}\}\cup \{s_i\}\cup\{x_i\}.
\]

The POP~\eqref{eq:pop_sparse} can equivalently be reformulated as
\begin{equation}
\label{eq:lifted-pop}
\begin{aligned}
\inf_{x,s}\quad & s_n \\
\text{s.t.}\quad 
& h_{i,l}(x_i,s_{i-1},s_{i})=0, \qquad i=1,\dots,n,\ l=1,\dots,r_i,\\
& g_i(x_i)\ge 0,\qquad i=1,\dots,n,
\end{aligned}
\end{equation}

Let us define, as we do for the other sparse hierarchies, the graph corresponding to the problem, such that we see where the cliques come from and how the constraints fit into these.

Let $G_{r}=(V,E)$ be the correlative sparsity graph of problem \eqref{eq:lifted-pop}, where the set of vertices is given by the set of variables, and two vertices are connected by an edge if they appear in the same constraint. The graph $G_{r}$ depends only on the state dimensions, or ranks, $r=\{r_1,\dots,r_{n-1},1\}$, and the number of variables, $n$, of $p(x)$. 

\begin{remark}
    For a fully general (dense) TT representation of the polynomial, as in Figure \ref{fig:graph}, the $s_i$ variables will be fully connected for each $i$, resulting in a chordal graph. If the tensor cores have sparse structures, $G_r$ might have less edges and lose the chordality, yielding smaller cliques after finding an adequate chordal extension.
\end{remark}

\begin{figure}[h]
    \centering
    \includegraphics[width=0.95\linewidth]{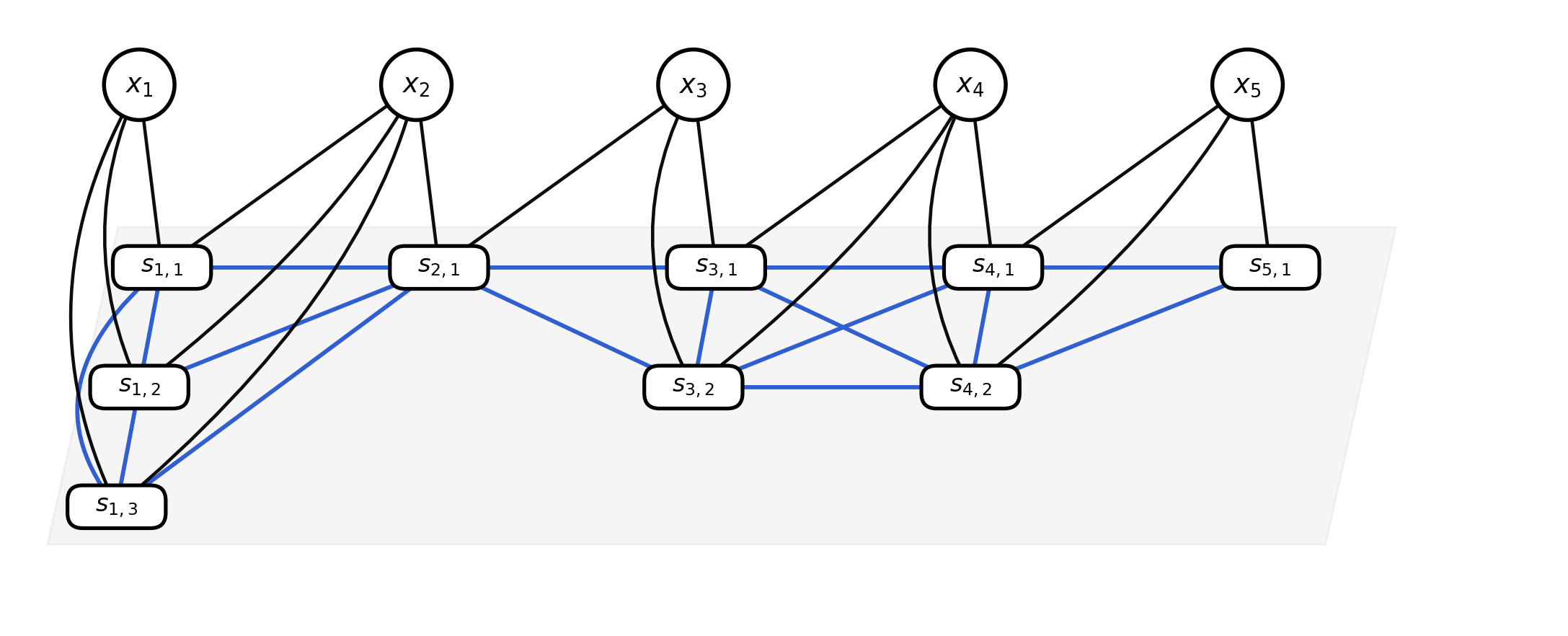}
    \caption{$G_r$ graph for $r=(3,1,2,2,1)$, where the black edges are the edges connecting original variables to lifted variables, and the blue lines connect lifted variables between themselves.}
    \label{fig:graph}
\end{figure}

\begin{figure}[h]
    \centering
    \includegraphics[width=0.95\linewidth]{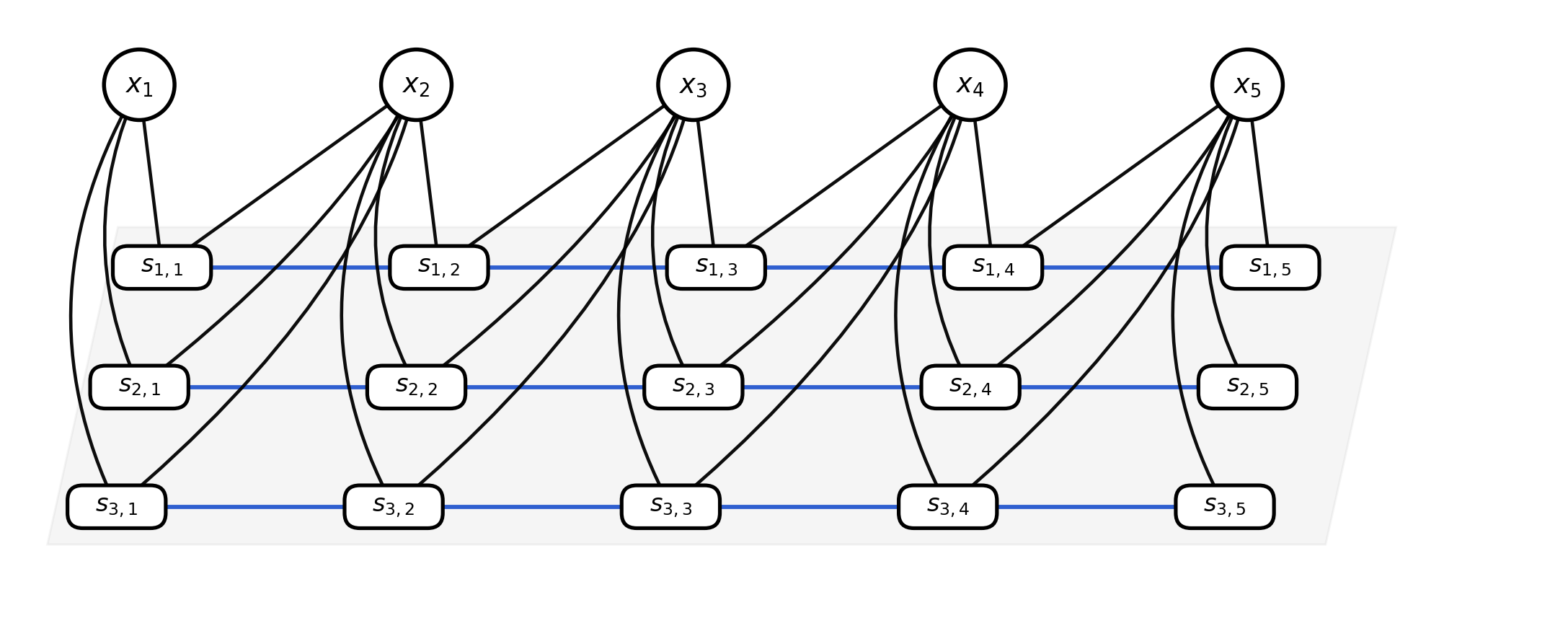}
    \caption{Graph for the CP-decomposed polynomial (as in \ref{eq:cp_poly}) of rank $3$ and in $5$ variables. The black edges are the edges connecting original variables to lifted variables, and the blue lines connect lifted variables between themselves.}
    \label{fig:graph_cp}
\end{figure}

We can compare the lifted correlative sparsity graph of the composition (and tensor train) form with that of the CP-decomposition form in Figure \ref{fig:graph_cp}. Note that the latter is not a chordal so it requires the addition of new chords in order to find a chordal extension and construct its cliques, while for the previous graph, we already have chordality.

A general polynomial with dense tensor cores will have a chordal graph $G_{r}$. Then there exists a clique tree decomposition satisfying the running intersection property (RIP). Note that by how we define the clique decomposition we have $n$ cliques $\mathcal{I}=\{I_1,\dots,I_n\}$ with $I_i$ containing the indices of variables of $\{x_i\} \cup \{s_{i}\} \cup \{s_{i-1}\}.$

By construction of $G_{r}$, every constraint polynomial $h_{i,l}$ and $g_i$ is supported on a single clique.
Therefore, we call $J_i \subset \{i\}$ the set of inequalities whose variables are contained in $I_i$, which are $g_i(x_i) \ge 0$, and $H_i = \{(i,1),\ldots,(i,r_i)\}$ the set of equalities whose variables are contained in $I_i$, $h_i$.

Let $z_{I_i}$ denote the vector of variables $\{x_i\} \cup \{s_{i}\} \cup \{s_{i-1}\}$ indexed by $I_i$. For a relaxation order $k$, let $y^{(i)}$ denote a truncated moment sequence in the variables $z_{I_i}$ up to degree $2k$, and let $M^{I_i}_k(y^{(i)})$ be the corresponding moment matrix. Similarly, let $M^{I_i}_{k-d_j}(g_j y^{(i)})$ be the localizing matrix for inequality $g_j$.

Since the set $\{x \mid g_i(x_i) \ge 0\}$ is compact and the functions $F_i$ continuous, there exist numbers $M_1,\ldots,M_n$ and $R_1,\ldots,R_n$ such that the ball constraints $x_i^2 \le M_i^2$ and $\|s_i\|^2 \le R_i^2$ can be added to~(\ref{eq:lifted-pop}) without changing its optimal value. This will ensure the convergence of the hierarchy, as shown in the next subsection.

The clique-based moment relaxation is given by

\begin{equation}\label{eq:slmom-sdp}
\begin{aligned}
(P^{\mathrm{chord}}_k):\quad 
p^{\mathrm{chord}}_k = \min_{\{y^{(i)}\}_{i=1}^n}& \quad 
 \sum_{i=1}^n L_{y^{(i)}}(p_i) \\
\text{s.t.}\quad 
& M^{I_i}_k(y^{(i)}) \succeq 0, \qquad i=1,\dots,n, \\
& M^{I_i}_{k-d_j}(g_j\,y^{(i)}) \succeq 0, 
   \qquad i=1,\dots,n,\ j\in J_i, \\
& M^{I_i}_{k-1}\bigl((M_i^2-x_i^2)\,y^{(i)}\bigr) \succeq 0,
   \qquad i=1,\dots,n, \\
& M^{I_i}_{k-1}\bigl((R_i^2-\|s_i\|^2)\,y^{(i)}\bigr) \succeq 0,
   \qquad i=1,\dots,n, \\
& L_{y^{(i)}}(q\,h_{i,l}) = 0, 
   \qquad (i,l)\in H_i,\ \deg(qh_{i,l})\leq 2k, \\
& y^{(a)}\big|_{S_{ab}} = y^{(b)}\big|_{S_{ab}}, 
   \qquad (a,b)\text{ edge of clique tree}, \\
& y^{(i)}_0 = 1, \qquad i=1,\dots,n.
\end{aligned}
\end{equation}

And the dual is 
\begin{equation}\label{eq:slsos-sdp}
\begin{aligned}
(D^{\mathrm{chord}}_k):\quad 
\lambda^{\mathrm{chord}}_k 
= &\sup_{\lambda,\{Q_{i,0}\},\{Q_{i,j}\},\{Q_{i,x}\},\{Q_{i,s}\},\{\tau_{i,l}\}}
\quad  \lambda \\
\text{s.t.}\quad 
& \sum_{i=1}^n p_i(x_{I_i}) - \lambda
  - \sum_{i=1}^n \sum_{j\in J_i} \sigma_{i,j}(x_{I_i}) g_j(x_{I_i}) \\
& \qquad
  - \sum_{i=1}^n \sigma_{i,x}(x_{I_i}) (M_i^2-x_i^2)
  - \sum_{i=1}^n \sigma_{i,s}(x_{I_i}) (R_i^2-\|s_i\|^2) \\
& \qquad
  - \sum_{i=1}^n \sum_{(i,l)\in H_i}
    \tau_{i,l}(x_{I_i}) h_{i,l}(x_{I_i})
  = \sum_{i=1}^n \sigma_{i,0}(x_{I_i}), \\
& \sigma_{i,0}(x_{I_i})
  = z_{k,I_i}(x)^\top Q_{i,0} z_{k,I_i}(x), \\
& \sigma_{i,j}(x_{I_i})
  = z_{k-d_j,I_i}(x)^\top Q_{i,j} z_{k-d_j,I_i}(x), \\
& \sigma_{i,x}(x_{I_i})
  = z_{k-1,I_i}(x)^\top Q_{i,x} z_{k-1,I_i}(x), \\
& \sigma_{i,s}(x_{I_i})
  = z_{k-1,I_i}(x)^\top Q_{i,s} z_{k-1,I_i}(x), \\
& Q_{i,0} \succeq 0,\quad Q_{i,j}\succeq 0,\quad Q_{i,x}\succeq 0,\quad Q_{i,s}\succeq 0, \\
& \deg(\tau_{i,l} h_{i,l}) \leq 2k .
\end{aligned}
\end{equation}
We call this moment-SOS hierarchy $\mathrm{SL}_{\mathrm{chord}}$, for \emph{state-lifting chordal} hierarchy.

\begin{remark}
    The extraction of minimizers is expected under the usual sparse flatness conditions from the sparse moment-SOS literature \cite{Waki2006,Lasserre2006,FantuzziFuentes2025}. If the clique moment matrices satisfy the appropriate rank conditions and are consistent on the separators, then one can recover globally compatible atoms by the standard sparse extraction machinery.
\end{remark}

\subsection{Convergence}

We show that the $\mathrm{SL}_{\mathrm{chord}}$ hierarchy converges to the global optimum. 

\begin{theorem}
We have
\[
p_k^{\mathrm{chord}} \nearrow \; p^*
\qquad \text{as } k\to\infty.
\]
\end{theorem}

\begin{proof}
The lifted formulation in variables $(x,s)$ is equivalent to the original compositional polynomial optimization problem, since the equalities
\[
s_1 = F_1(x_1), \qquad s_i = F_i(s_{i-1},x_i), \quad i=2,\dots,n,
\]
uniquely determine the lifting variables from $x$. Therefore both problems have the same global optimum $p^*$.
The constants $M_i$ and $R_i$ have been chosen such that the ball constraints 
\[
M_i^2-x_i^2\ge 0, \qquad R_i^2-\|s_i\|^2\ge 0, \qquad i=1,\dots,n,
\]
are redundant and hence do not change the optimal value of the polynomial optimization problem.

For each clique $I_i$, let $\mathcal Q_i$ denote the quadratic module generated by the
inequalities assigned to that clique, including the redundant ball constraints above. Define
\[
q_i(z_{I_i}) := (M_i^2-x_i^2) + (R_i^2-\|s_i\|^2), \qquad i=1,\dots,n.
\]
Since each $q_i$ is a sum of generators of $\mathcal Q_i$ with SOS coefficients equal to $1$,
we have
\[
q_i \in \mathcal Q_i, \qquad i=1,\dots,n.
\]
Summing over all cliques yields
\[
\sum_{i=1}^n q_i(z_{I_i})
= \sum_{i=1}^n (M_i^2-x_i^2) + \sum_{i=1}^n (R_i^2-\|s_i\|^2)
= R^2-\|x\|^2-\|s\|^2,
\]
where
\[
R^2 := \sum_{i=1}^n M_i^2 + \sum_{i=1}^n R_i^2.
\]
Hence, the sparse Archimedean condition holds.

The clique decomposition used satisfies the running intersection property by construction, and the equalities of the original problem are handled through the local ideals they generate. Therefore, the standard convergence theorem for sparse moment-SOS hierarchies applies, and yields monotone convergence of the clique-based relaxations to the optimum of the lifted problem; see \cite{Lasserre2006,Waki2006}. Hence
$\
p_k^{\mathrm{chord}} \nearrow p^*.
$
\end{proof}

\subsection{Treewidth}

The structure of the lifted graph $G_{r}$ is completely determined by the state dimensions $r=\{r_1,\dots,r_{n-1},1\}$ and the number of variables $n$. Since the size of the PSD blocks in the clique-based hierarchy coincides with the size of the largest clique in $G_{r}$, we must quantify this size precisely.

To this end, we recall the notion of a tree decomposition and treewidth.

A \emph{tree decomposition} of a graph $G=(V,E)$ consists of a tree $T$ together with a family of subsets (called \emph{bags}) $X_1,\dots,X_L \subseteq V$ satisfying: 
\begin{itemize}
    \item $\bigcup_{i=1}^L X_i = V$;
    \item for every edge $(v,w)\in E$, there exists a bag $X_i$ containing both $v$ and $w$;
    \item for every vertex $v\in V$, the collection of bags containing $v$ forms a connected subtree of $T$.
\end{itemize}

The \emph{width} of a tree decomposition is $\max_i |X_i|-1$, and the \emph{treewidth} of $G$, denoted $\mathrm{tw}(G)$, is the minimum width over all tree decompositions of $G$.

For the clique-based moment-SOS construction, $G_{r}$ is already chordal, so it induces a clique tree decomposition, and we know that every tree decomposition of minimal width corresponds to a chordal completion whose maximal clique size equals $\mathrm{tw}(G_{r})+1$. Therefore,
\[
\max \text{clique size of } G_{r}\
=
\mathrm{tw}(G_{r})+1.
\]

Controlling the treewidth of $G_{r}$ therefore directly controls the
size of the largest PSD block in the hierarchy.

\begin{theorem}
    The variable graph $G_{r}$ induced by a compositional polynomial with state dimension $r=(r_1,\dots,r_n)$ has treewidth
    \[
    \mathrm{tw}(G_{r,n})=\max_i\{r_i+r_{i+1}\}.
    \]
    The treewidth is totally independent of the number of variables $n$.
    \label{TW}
\end{theorem}

\begin{proof}
    We prove the theorem by first using a chordal completion argument to get the upper bound, and a graph minor argument to obtain the lower bound.

    \textbf{Upper bound.} Consider the elimination order that goes column by column from right to left, first by eliminating the lifting variables $s_{i,l}$, and then the $x_i$ variables. The total ordering will be $s_{n,1},x_n,s_{n-1,1},\dots,s_{n-1,r_{n-1}},x_{n-1},\dots,s_{1,r_1},x_1$. Note that at each elimination step all the neighbors form a complete graph, so the ordering will be a PEO. During the ordering also notice how the biggest clique is of size $\max_i \{r_i+r_{i+1}+1\}$, for some $i$. Therefore we have an upper bound for the treewidth of $\max_i \{r_i+r_{i+1}\}$.

    \textbf{Lower bound.} Take the index $i$ which attains $\max_i \{r_i+r_{i+1}\}$, and consider the graph minor formed by the lifting variables in columns $i$ and $i+1$, and the variable $x_{i+1}$. This graph minor forms a complete graph with treewidth $r_i+r_{i+1}$, therefore we have a lower bound for the full graph treewidth, $\max_i \{r_i+r_{i+1}\}$.

    Putting both the upper and lower bounds together we show $\mathrm{tw}(G_{r,n})=\max_i\{r_i+r_{i+1}\}$.
\end{proof}

\begin{remark}
    The theorem is stated and proved for the fully general dense case, in which the state variables are fully coupled across neighboring stages and, in the TT case, the tensor cores have no additional sparsity. If some state variables are decoupled, or if the tensor cores are sparse, then the induced graph can have smaller treewidth. A clear example of this is letting the tensor cores be diagonal and size $r$, which then leads to a treewidth of $2r$, and this is equivalent to a CP decomposition of rank $r$, with treewidth $r+1\leq2r$ as we show in \cite{LRPOP25}. Therefore, additional structure in the composition maps can lead to smaller treewidths and smaller moment matrices.
\end{remark}

\subsection{Complexity}

Now we study how the $\mathrm{SL}_{\mathrm{chord}}$ hierarchy scales with the different parameters of the problem. Throughout this section we will assume all ranks are equal, $r_i=r\ \forall i$, and we start by considering the biggest cliques of the problem, which will be 
\[
    I_i=\{s_{i-1},x_i,s_i\}, \text{ with }\quad |I_i|= 2r+1.
\]

Now we can consider the main two complexity bottlenecks of SDP solvers, which are the size of the SDP blocks and the number of equality constraints.

\textbf{SDP block sizes.} For $\mathrm{SL}_{\mathrm{chord}}$ relaxation order $k$ we will have the biggest SDP blocks of size
\[
\binom{|I_i|+k}{k}=\binom{2r+1+k}{k},
\]
and there will be $O(n)$ many of such blocks. Therefore, the size of the biggest block will not depend on the total number of variables, just on the ranks.

\textbf{Overlap equalities.} By our definition of the cliques, we see that the overlap between two adjacent variable bags is of size $r$. To enforce consistency with our moment matrices, we equate all the moments up to degree $2k$ on the separator. This results in a total of
\[
(n-1)\binom{r+2k}{2k}
\]
equalities arising from the overlaps between cliques. As we see this is linear in $n$.

\textbf{Lifting equalities.} Every equality $h_{i,l}(x_i,s_{i-1},s_i)=0$ is imposed through 
\[
L_{y^{(i)}}(qh_{i,l})=0,
\]
for all monomials $q$ satisfying $\deg(qh_{i,l})\leq 2k$. Let $d=\deg(h_{i,l})$, then the number of such monomials in clique $I_i$ is 
\[
\binom{|I_i|+(2k-d)}{2k-d}=\binom{2r+1+2k-d}{2k-d}.
\]
Therefore, we have a total number of lifting equalities equal to 
\[
nr\binom{2r+1+2k-d}{2k-d}.
\]
This is also linear in $n$. Then we have a total number of equalities 
\[
(n-1)\binom{r+2k}{2k}+nr\binom{2r+1+2k-d}{2k-d}.
\]

Thus, for fixed $r, k$, and local degree d, the total number of equality constraints grows linearly with $n$. Their dependence on $r$ and $d$, however, is combinatorial through the binomial coefficients above, so increasing either the rank or the local degree can rapidly increase the size of the resulting SDP. In particular, for moderate or large $r$, the equality constraints are expected to become the main computational bottleneck of the $\mathrm{SL}_{\mathrm{chord}}$ hierarchy.

\section{State-lifting push-forward hierarchy}\label{sec:push-forward}

Now we consider a different type of moment-SOS hierarchy native to the compositional structure of the polynomials we study. In this case we encode the information of the chain structure at the level of the measures directly \cite{lasserre2008nonlinear}. As before we want to solve the POP~\eqref{eq:pop_sparse}, assuming that $p$ is represented as
\[
s_1 =F_1(x_1), \quad  s_i  = F_i(s_{i-1},x_i), \quad p(x)=s_n,
\]
where $s_i\in \R^{r_i}$ and $r_0=r_n=1$. Let us introduce the lifting variables $s_{i,l}$ for $i=1,\dots,n$ and $l=1,\dots,r_i$. 

We introduce the transition measure at step $i$, $\mu_i$, on variables $(s_{i-1},x_i)$, which encodes how the previous state and the current variable are distributed as well as an initial measure $\mu_1$ on the variable $x_1$. Now define for measurable maps $F_i$ and measures $\mu_i$, the push-forward of $\mu_i$ by $F_i$ to be $(F_i)_{\#}\mu_i$, which will be an image measure. We also let the projection map $\pi_{s_{i-1}}:(s_{i-1},x_i) \mapsto s_{i-1}$ be such that $(\pi_{s_{i-1}})_\#\mu_i$ is the marginal of $\mu_i$ on $s_{i-1}$. To maintain consistency, the state $s_{i-1}$ in the measure $\mu_i$ must match $s_{i-1}$ in the push-forward $(F_{i-1})_{\#}\mu_{i-1}$. 

As for the $\mathrm{SL}_{\mathrm{chord}}$ hierarchy, there exist numbers $M_1,\ldots,M_n$ and $R_1,\ldots,R_n$ such that the ball constraints $x_i^2 \le M_i^2$ and $\|s_i\|^2 \le R_i^2$ can be added to~(\ref{eq:measure-pop}) without changing its optimal value. This ensures the convergence of the hierarchy, as shown in the next subsection.

Now we write the POP in terms of measures as follows.

\begin{equation}\label{eq:measure-pop}
\begin{aligned}
\inf_{\{\mu_i\}} \quad
& \int F_n(s_{n-1},x_n) \, d\mu_n(s_{n-1},x_n)\\
\text{s.t.}\quad
& (\pi_{s_{i-1}})_\#\mu_i\;=\; (F_{i-1})_{\#}\mu_{i-1}, \quad i=2,\dots,n,\\
& \operatorname{supp}(\mu_1)\subset \{x_1:\ g_1(x_1)\ge 0,\ M_1^2-x_1^2\ge 0\},\\
& \operatorname{supp}(\mu_i)\subset \{(s_{i-1},x_i):\ g_i(x_i)\ge 0,\ M_i^2-x_i^2\ge 0,\ R_{i-1}^2-\|s_{i-1}\|^2\ge 0\},\quad i=2,\dots,n,\\
& \mu_1\in\mathcal{M}_+,\ \mu_1(\R)=1,   \\
& \mu_i\in\mathcal{M}_+,\ \mu_i(\R^{r_{i-1}}\times\R)=1, \quad i=2,\dots,n,
\end{aligned}
\end{equation}

where $\mathrm{supp}(\mu)$ is the support of the measure $\mu$ and $\mathcal{M}_+$ denotes the set of nonnegative Borel measures. For the hierarchy of SDPs, fix a relaxation order $k\in\mathbb{N}$.
For each stage $i=1,\dots,n$, let $y^{(i)}$ denote a truncated moment sequence of order $2k$
associated with the transition measure $\mu_i$:
\[
\mu_1 \ \text{on}\ x_1, 
\qquad 
\mu_i \ \text{on}\ (s_{i-1},x_i)\ \ (i\ge 2).
\]
Let $M_k(\cdot)$ denote the moment matrix and $M_{k-d_i}(g_i\,\cdot)$ the localizing matrix,
where $d_i:=\lceil \deg(g_i)/2\rceil$.

For multi-indices $\alpha\in\mathbb{N}^{r_i}$, write $s_i^\alpha:=\prod_{\ell=1}^{r_i} s_{i,\ell}^{\alpha_\ell}$.
Define the push-forward monomials
\[
\big(F_i(s_{i-1},x_i)\big)^\alpha := \prod_{\ell=1}^{r_i} F_{i,\ell}(s_{i-1},x_i)^{\alpha_\ell},
\]
and for $i=1$ interpret $F_1(x_1)^\alpha$ similarly.

The push-forward (measure propagation) relaxation is the SDP
\begin{equation}\label{eq:push-forward-primal}
\begin{aligned}
(P_k^{\mathrm{push}}):\quad 
p_k^{\mathrm{push}}& \;=\; \inf_{\{y^{(i)}\}}
\quad L_{y^{(n)}}(F_n(s_{n-1},x_n)) \\
\text{s.t.}\quad 
& M_k\!\big(y^{(i)}\big)\succeq 0,\quad i=1,\dots,n, \\
& M_{k-1}\bigl((M_i^2-x_i^2)\,y^{(i)}\bigr) \succeq 0,\qquad i=1,\dots,n,\\
& M_{k-1}\bigl((R_{i-1}^2-\|s_{i-1}\|^2)\,y^{(i)}\bigr) \succeq 0,\qquad i=2,\dots,n,\\
& M_{k-d_i}\!\big(g_i\,y^{(i)}\big)\succeq 0,\quad i=1,\dots,n, \\
& L_{y^{(2)}}\!\big(s_{1}^\alpha\big) = L_{y^{(1)}}\!\big(F_1(x_1)^\alpha\big),
\, \forall \alpha\in\mathbb{N}^{r_1},\ \deg(F_1^{\alpha})\le 2k,\\
& L_{y^{(i+1)}}\!\big(s_{i}^\alpha\big) = L_{y^{(i)}}\!\big(F_i(s_{i-1},x_i)^\alpha\big),
\quad \forall \alpha\in\mathbb{N}^{r_i},\ \deg(F_i^{\alpha})\le 2k,\ \ i=2,\dots,n-1,\\
& y^{(i)}_{0}=1,\quad i=1,\dots,n.
\end{aligned}
\end{equation}
Here $L_{y^{(i)}}$ denotes the Riesz functionals acting on polynomials in the
corresponding variables. 

Let $w_{k,1}(x_1)$ denote the vector of monomials in $x_1$ of degree at most $k$.
For $i=2,\dots,n$, let $w_{k,i}(s_{i-1},x_i)$ denote the vector of monomials in $(s_{i-1},x_i)$
of degree at most $k$, and let $w_{k-d_i,i}(s_{i-1},x_i)$ denote the vector of monomials in
$(s_{i-1},x_i)$ of degree at most $k-d_i$.
For $i=1$, we analogously write $w_{k-d_1,1}(x_1)$ and $w_{k-1,1}(x_1)$. Introduce potential polynomials $V_i \in \mathbb{R}[s_i]$ of degree $\le 2k$ for
$i=1,\dots,n-1$ (with the convention $V_n(s_n):=s_n$), and SOS multipliers
$\sigma_{i,0},\sigma_{i,1}$ for each step constraint $g_i(x_i)\ge 0$.

The dual push-forward relaxation can be written as
\begin{equation}\label{eq:push-forward-dual-gram}
\begin{aligned}
(D^{\mathrm{push}}_k):\quad 
\lambda^{\mathrm{push}}_k
= &\sup_{\lambda,\{V_i\},\{Q_{i,0},Q_{i,1},Q_{i,x}\},\{Q_{i,s}\}_{i=2}^n}
\quad  \lambda\\
\text{s.t.}\quad 
& V_1(F_1(x_1)) - \lambda
 = \sigma_{1,0}(x_1) + \sigma_{1,1}(x_1) g_1(x_1)
 + \sigma_{1,x}(x_1)(M_1^2-x_1^2),\\
& V_i(F_i(s_{i-1},x_i)) - V_{i-1}(s_{i-1})
 = \sigma_{i,0}(s_{i-1},x_i)
 + \sigma_{i,1}(s_{i-1},x_i) g_i(x_i)\\
& \hspace{4.3cm}
 + \sigma_{i,x}(s_{i-1},x_i)(M_i^2-x_i^2)\\
& \hspace{4.3cm}
 + \sigma_{i,s}(s_{i-1},x_i)(R_{i-1}^2-\|s_{i-1}\|^2),\quad i=2,\dots,n,\\
& \sigma_{1,0}(x_1) = w_{k,1}(x_1)^\top Q_{1,0}\, w_{k,1}(x_1),\\
& \sigma_{i,0}(s_{i-1},x_i) = w_{k,i}(s_{i-1},x_i)^\top Q_{i,0}\, w_{k,i}(s_{i-1},x_i),
\qquad i=2,\dots,n,\\
& \sigma_{1,1}(x_1) = w_{k-d_1,1}(x_1)^\top Q_{1,1}\, w_{k-d_1,1}(x_1),\\
& \sigma_{i,1}(s_{i-1},x_i) = w_{k-d_i,i}(s_{i-1},x_i)^\top Q_{i,1}\, w_{k-d_i,i}(s_{i-1},x_i),
\qquad i=2,\dots,n,\\
& \sigma_{1,x}(x_1) = w_{k-1,1}(x_1)^\top Q_{1,x}\, w_{k-1,1}(x_1),\\
& \sigma_{i,x}(s_{i-1},x_i) = w_{k-1,i}(s_{i-1},x_i)^\top Q_{i,x}\, w_{k-1,i}(s_{i-1},x_i),
\qquad i=2,\dots,n,\\
& \sigma_{i,s}(s_{i-1},x_i) = w_{k-1,i}(s_{i-1},x_i)^\top Q_{i,s}\, w_{k-1,i}(s_{i-1},x_i),
\qquad i=2,\dots,n,\\
& Q_{i,0} \succeq 0,\ Q_{i,1} \succeq 0,\ Q_{i,x} \succeq 0,\qquad i=1,\dots,n,\\
& Q_{i,s} \succeq 0,\qquad i=2,\dots,n,
\end{aligned}
\end{equation}

The role of the potential polynomials $V_i$ becomes clear by summing the
step-by-step identities in~\eqref{eq:push-forward-dual-gram}.
For $i=1$ we have
\[
V_1(F_1(x_1)) - \lambda
=
\sigma_{1,0}(x_1) + \sigma_{1,1}(x_1) g_1(x_1)
+ \sigma_{1,x}(x_1)(M_1^2-x_1^2),
\]
and for $i=2,\dots,n$,
\begin{align*}
V_i(F_i(s_{i-1},x_i)) - V_{i-1}(s_{i-1})
=&\,
\sigma_{i,0}(s_{i-1},x_i)
+\sigma_{i,1}(s_{i-1},x_i) g_i(x_i)\\
&+\sigma_{i,x}(s_{i-1},x_i)(M_i^2-x_i^2)
+\sigma_{i,s}(s_{i-1},x_i)(R_{i-1}^2-\|s_{i-1}\|^2).
\end{align*}
Summing these equalities over $i=1,\dots,n$ yields a telescoping cancellation of the
intermediate terms $V_{i-1}(s_{i-1})$, leaving
\begin{align*}
s_n-\lambda
=&  \sigma_{1,0}(x_1)+\sigma_{1,1}(x_1)g_1(x_1)+\sigma_{1,x}(x_1)(M_1^2-x_1^2) \\
&+\sum_{i=2}^n \sigma_{i,0}(s_{i-1},x_i)
+\sum_{i=2}^n \sigma_{i,1}(s_{i-1},x_i)g_i(x_i)\\
&+\sum_{i=2}^n \sigma_{i,x}(s_{i-1},x_i)(M_i^2-x_i^2)
+\sum_{i=2}^n \sigma_{i,s}(s_{i-1},x_i)(R_{i-1}^2-\|s_{i-1}\|^2).
\end{align*}
Hence $s_n-\lambda$ admits a decomposition as a sum of SOS polynomials plus SOS multiples of the constraints $g_i$ and of the redundant ball constraints, showing that $\lambda$ is a valid lower bound on the global optimum.

Note that, while in $\mathrm{SL}_{\mathrm{chord}}$ we had $\max_i r_i+r_{i+1}$ variables per moment matrix, which came from the maximum clique size of the correlative sparsity graph, now we only have $\max_i r_i+1$ variables per moment matrix. In this case, we are not building a correlative sparsity graph but we are grouping the variables according to the measures we define, which result in the different moment matrices.

However, now the degree $\alpha$ enters the picture, if this is small we might have smaller SDP blocks, and if it is bigger we will get bigger SDP blocks than we previously had. The degree we pick will depend on how the convergence of the hierarchy behaves, for some examples it might converge for low degree, and maybe in other cases it might require higher degree.

We will call this moment-SOS hierarchy $\mathrm{SL}_{\mathrm{push}}$, for \emph{state-lifting push-forward hierarchy}.

\subsection{Convergence}

To prove that $\mathrm{SL}_{\mathrm{push}}$ converges to the optimal value, we start by proving that the measure problem is equivalent to the original compositional polynomial problem. For this we will need the so-called Gluing Lemma.

\begin{lemma}[Gluing lemma]\label{lemma:gluing}
Let $\mu_1$ be a probability measure on $X_1\times X_2$, and let
$\mu_2$ be a probability measure on $X_2\times X_3$. If their marginals on $X_2$ coincide, that is
\[
(\pi_{X_2})_{\#}\mu_1 = (\pi_{X_2})_{\#}\mu_2,
\]
then there exists a probability measure $\mu_T$ on $X_1\times X_2\times X_3$ with
\[
(\pi_{X_1,X_2})_{\#}\mu_T = \mu_1,\quad (\pi_{X_2,X_3})_{\#}\mu_T = \mu_2.
\]
\end{lemma}

This is a standard result, see for example \cite[Lem. B.13, p.~308]{Lasserre2010}.

\begin{proposition}
Assume each local feasible set $K_i$ is compact. Then the infinite-dimensional $\mathrm{SL}_{\mathrm{push}}$ measure formulation has the same optimal value as the original compositional POP.
\label{prop:gluing}
\end{proposition}

\begin{proof}
We prove both inequalities.

First, let $(x_1,\dots,x_n)$ be any feasible point of the original problem, and define the states at each step recursively by
\[
s_1=F_1(x_1), \qquad s_i=F_i(s_{i-1},x_i), \quad i=2,\dots,n.
\]
This feasible trajectory induces a family of Dirac measures:
\[
\mu_1=\delta_{x_1}, \quad \mu_i=\delta_{(s_{i-1},x_i)}.
\]
for $i=2,\dots, n$. These measures satisfy the $\mathrm{SL}_{\mathrm{push}}$ support, marginal consistency constraints, and the objective value is exactly $s_n$. Therefore, the $\mathrm{SL}_{\mathrm{push}}$ measure optimum is at most the optimum of the original problem.

Conversely, let $(\mu_i)_{i=1}^n$ be any feasible family for the $\mathrm{SL}_{\mathrm{push}}$ measure formulation. By the marginal consistency constraints,
\[
(\pi_{s_{i-1}})_\#\mu_i\;=\; (F_{i-1})_{\#}\mu_{i-1}, \qquad i=2,\dots,n.
\]
Thus the local measures are consistent on their overlaps. The problem has a chain structure, and the overlap of all local blocks is exactly the state variable $s_{i-1}$, therefore by imposing the marginal consistency constraints we are ensuring that these shared variables are the same if we look at them from one stage or from the next one. Specifically, we have $\mu_1(x_1)$ and $\mu_2(s_1,x_2)$ with the marginals satisfying $(\pi_{s_{1}})_\#\mu_2= (F_{1}(x_1))_{\#}\mu_{1}$. Now we can define the image of $\mu_1$ via the map $x_1\mapsto (x_1,F_1(x_1))$ to define the auxiliary measure $\hat{\mu}_1(x_1,s_1)$. Similarly, via the map $(s_1,x_2)\mapsto (s_1,x_2,F_2(s_1,x_2))$ we define the auxiliary measure $\hat{\mu}_2(s_1,x_2,s_2)$, where $(\pi_{s_{1}})_\#\hat{\mu}_1= (\pi_{s_{1}})_\#\hat{\mu}_2$, which allows to glue them through their overlap $s_1$, to get the measure $\hat{\mu}_{(1,2)}(x_1,s_1,x_2,s_2)$.

We can repeat this process consecutively, the next step would be taking $\hat{\mu}_{(1,2)}(x_1,s_1,x_2,s_2)$ and $\hat{\mu}_3(s_2,x_3,s_3)$, satisfying the marginal condition $(\pi_{s_{2}})_\#\hat{\mu}_{(1,2)}= (\pi_{s_{2}})_\#\hat{\mu}_3$, which yields a new measure $\hat{\mu}_{(1,2,3)}(x_1,s_1,x_2,s_2,x_3,s_3)$. After repeating this process a total of $n-1$ times we obtain a global measure $\hat{\mu}_T(x_1,s_1,x_2,\dots,s_n)$.

Moreover, by the support and push-forward constraints, this global measure is supported on the feasible set of the original compositional problem. Hence, the $\mathrm{SL}_{\mathrm{push}}$ objective is an average of feasible objective values of the original problem, so it is bounded below by the global optimum. Therefore, the $\mathrm{SL}_{\mathrm{push}}$ measure optimum is at least the optimum of the original problem. Combining both inequalities proves the claim.
\end{proof}

Now we have to show that the $\mathrm{SL}_{\mathrm{push}}$ hierarchy does indeed converge to the global optimum.

\begin{theorem}
We have
\[
p_k^{\mathrm{push}} \nearrow p^*
\qquad \text{as } k\to\infty.
\]
\end{theorem}

\begin{proof}

By Proposition~\ref{prop:gluing}, the infinite-dimensional
$\mathrm{SL}_{\mathrm{push}}$ measure problem is an exact reformulation of the original
compositional polynomial optimization problem, and therefore has optimal value $p^*$.

We are thus in the setting of a generalized moment problem with finitely many measures,
linear coupling constraints, and compact Archimedean support sets (due to the redundant ball constraints). Standard convergence
results for Lasserre-type relaxations of compact generalized moment problems then imply that
the associated SDP relaxations converge monotonically to the value of the measure problem \cite{Lasserre2008GPM}.
Hence
\[
p_k^{\mathrm{push}} \nearrow p^* .
\]
\end{proof}

\subsection{Complexity}

To study how the $\mathrm{SL}_{\mathrm{push}}$ hierarchy scales in terms of computational resources needed to solve the problem, we start by assuming all ranks are the same $r_i=r \ \forall i$. The  measures $\mu_i$ live on $(s_{i-1},x_i)$ of dimension $r+1$, and will have moment matrices of relaxation order $k_{\mu}$. We also let the degree of the composition map to be $\deg(F_{i,l})=d$.

\textbf{SDP block sizes.} The sizes of the dominant moment matrices are
\[
\binom{r+1+k_{\mu}}{k_{\mu}}, 
\]
and the total number of SDP blocks is $n$.

\textbf{Push-forward equalities.} The $\mathrm{SL}_{\mathrm{push}}$ hierarchy enforces $(\pi_{s_{i}})_\#\mu_{i+1} = (F_{i})_{\#}\mu_{i}$, which at the moment level implies equalities
\[
L_{y^{(i+1)}}\!\big(s_{i}^\alpha\big) = L_{y^{(i)}}\!\big(F_i(s_{i-1},x_i)^\alpha\big).
\]
We can only enforce these if the moments are properly defined on both sides, which requires to fix a maximal degree $\alpha_{\max}$ with $\deg(F_i^{\alpha})\leq d|\alpha| \leq 2k_{\mu}$, which is equivalent to requiring $|\alpha|\leq \lfloor \frac{2k_{\mu}}{d}  \rfloor$. Therefore, the number of push-forward equalities is 
\[
(n-1)\binom{r+\lfloor \frac{2k_{\mu}}{d}  \rfloor}{\lfloor \frac{2k_{\mu}}{d}  \rfloor}.
\]

We see how for $\mathrm{SL}_{\mathrm{push}}$ the SDP block sizes are likely to be the ones defining the complexity, as the relaxation order $k_{\mu}$ might be very large. Meanwhile, for $\mathrm{SL}_{\mathrm{chord}}$ we saw how the number of equalities was the main complexity driver.

\section{Numerical examples}
\label{sec:examples}

In this section we study numerically how each hierarchy, $\mathrm{SL}_{\mathrm{chord}}$ and $\mathrm{SL}_{\mathrm{push}}$, gives bounds for optimization of polynomials with composition structure. Each of the hierarchies will provide a substantial scalability advantage for different types of composition polynomials. First, we compare the two proposed hierarchies with the dense hierarchy, and afterwards we study the computational limits for $\mathrm{SL}_{\mathrm{chord}}$ and $\mathrm{SL}_{\mathrm{push}}$ in terms of number of variables. We design a very specific type of tensor train polynomials for this example that mitigates numerical conditioning problems and therefore allows us to study the computational scalability in isolation.

We use the box constraints $x_i^2\leq 1$ to keep the inequalities local to the cliques. The numerical experiments were performed with the implemented hierarchies\footnote{https://github.com/llorebaga/SLPOP} using the \textit{Julia} \cite{Julia-2017} package \textit{TSSOS} \cite{Magron2021}. The SDP solver \textit{Mosek} \cite{Mosek} package was also a fundamental part of the code. The computer was a laptop with a CPU AMD Ryzen 7 PRO 7840U w/ Radeon 780M, 8 cores, base speed 3.30 GHz, and 32.0 GB of RAM.

\subsection{Comparison of composition and dense hierarchies}

Let us start considering random polynomials with generic quadratic composition structure, of the form
\begin{align*}
    s_{1,l}&=c_0^{(l)} + c_1^{(l)}x_1+c_2^{(l)} x_1^2\\
    s_{i,l}&=\sum_{p_1,p_2,\alpha=0}^2 c^{(i,l)}_{p_1,p_2,\alpha} s_{i-1,1}^{p_1}s_{i-1,2}^{p_2} x_i^{\alpha},
\end{align*}
where $l\leq r$ for rank $r=2$, degree 2 also in the local variables, as we see in the state definition, and we sample the random coefficients 
as
\[
c_{p_1,p_2,\alpha}^{(i,l)}=\frac{1}{2}\mathcal{N}(0,1)(0.7)^{p_1+p_2+\alpha+1},
\]
which will ensure the problem remains well conditioned and we can study the performance of the hierarchies independently from SDP solver errors.
For the numerics of Table \ref{tab:numerics_composition} we use a relaxation order of $k_{\text{dense}}=2^{n-2}$, which is the minimum one, considering the degree of the full polynomial is $2^{n-1}$. We have used relaxation order of $k_{\text{chord}}=3$ for $\mathrm{SL}_{\mathrm{chord}}$ and $k_{\mu}=4$ for $\mathrm{SL}_{\mathrm{push}}$. During the following examples we will set a maximum running time of $300$ seconds, and if we exceed this we will leave a blank space in the tables.

\begin{table}[h]
\centering
\begin{tabular}{l@{\hspace{2em}}ccccc}
\toprule
 & \multicolumn{5}{c}{$n$} \\
\cmidrule(lr){2-6}
 & 2 & 3 & 4 & 5 & 6 \\
\midrule

\multicolumn{6}{l}{\textbf{Optimum}} \\
Dense  & -0.120423 & 0.028422 & -0.240189 &- & - \\
$\mathrm{SL}_{\mathrm{chord}}$  & -0.120423 & 0.028422 & -0.240188 & -0.018408 & -0.503042 \\
$\mathrm{SL}_{\mathrm{push}}$  & -0.120423 & 0.028422 & -0.240188 & -0.018408 & -0.575960 \\
\addlinespace

\multicolumn{6}{l}{\textbf{Time (s)}} \\
Dense  & 0.004 & 0.053 & 2.257 & - & - \\
$\mathrm{SL}_{\mathrm{chord}}$  & 1.141 & 1.938 & 3.093 & 3.141 & 4.703 \\
$\mathrm{SL}_{\mathrm{push}}$  & 0.172 & 0.594 & 0.515 & 0.562 & 1.312 \\
\addlinespace

\multicolumn{6}{l}{\textbf{Block size}} \\
Dense  & 3 & 20 & 71 & - & - \\
$\mathrm{SL}_{\mathrm{chord}}$  & 56 & 56 & 56 & 56 & 56 \\
$\mathrm{SL}_{\mathrm{push}}$  & 35 & 35 & 35 & 35 & 35 \\
\addlinespace

\multicolumn{6}{l}{\textbf{Constraints}} \\
$\mathrm{SL}_{\mathrm{chord}}$  & 352 & 633 & 914 & 1195 & 1476 \\
$\mathrm{SL}_{\mathrm{push}}$  & 17 & 33 & 49 & 65 & 81 \\
\bottomrule
\end{tabular}
\caption{Comparison of Dense, $\mathrm{SL}_{\mathrm{chord}}$ and $\mathrm{SL}_{\mathrm{push}}$ for random quadratic composition polynomials of ranks 2 and local degree 2.}
\label{tab:numerics_composition}
\end{table}

We observe how the dense hierarchy starts taking several minutes for more than 4 variables, while $\mathrm{SL}_{\mathrm{chord}}$ and $\mathrm{SL}_{\mathrm{push}}$ remain within less than 10 seconds for $n=6$. A key observation is how the SDP block sizes of the dense hierarchy increase very fast while the ones of the $\mathrm{SL}_{\mathrm{chord}}$ and $\mathrm{SL}_{\mathrm{push}}$ remain constant, as expected. The driver of complexity is then the number of linear constraints, which increases linearly for both hierarchies. 

Let us do a similar study, but now we will consider tensor train polynomials, and furthermore we want to test the strengths and weaknesses of each of the proposed hierarchies. We will start with a tensor train polynomial of larger degree and lower rank, and then we will consider one with larger rank and lower degree, to see which hierarchy is suitable in which cases.

Firstly, we let $r=2,d=4$, and we fix the relaxation orders to be $k_{\text{dense}}=2n$, $k_{\text{chord}}=3$ and $k_{\mu}=10$, because if we try to use smaller relaxation orders the bounds are not tight. In Table \ref{tab:numerics_r2_d4} we see how $\mathrm{SL}_{\mathrm{push}}$ cannot find the optimal solution within the given time, due to the large SDP blocks it has to work with due to the high degree caused by explicit polynomial composition. On the other hand, $\mathrm{SL}_{\mathrm{chord}}$ can compute the optimum up to $n=6$ within 5 minutes and with results matching the dense hierarchy at the first examples; afterwards the dense hierarchy also struggles with the computational resources needed.

\begin{table}[h]
\centering
\begin{tabular}{l@{\hspace{2.5em}}ccccc}
\toprule
 & \multicolumn{5}{c}{$n$} \\
\cmidrule(lr){2-6}
 & 2 & 3 & 4 & 5 & 6 \\
\midrule

\multicolumn{6}{l}{\textbf{Optimum}} \\
Dense  & -0.659361 & -0.504356 & -0.378924 & - & - \\
$\mathrm{SL}_{\mathrm{chord}}$  & -0.659361 & -0.504356 & -0.378924 & -0.556061 & -0.399084 \\
$\mathrm{SL}_{\mathrm{push}}$  & - & - & - & - & - \\
\addlinespace

\multicolumn{6}{l}{\textbf{Time (s)}} \\
Dense  & 0.008 & 0.579 & 254.3 & - & - \\
$\mathrm{SL}_{\mathrm{chord}}$  & 3.219 & 52.67 & 81.76 & 120.7 & 165.3 \\
$\mathrm{SL}_{\mathrm{push}}$  & - & - & - & - & - \\
\addlinespace

\multicolumn{6}{l}{\textbf{Block size}} \\
Dense  & 5 & 31 & 119 & - & - \\
$\mathrm{SL}_{\mathrm{chord}}$  & 126 & 126 & 126 & 126 & 126 \\
$\mathrm{SL}_{\mathrm{push}}$  & - & - & - & - & - \\
\addlinespace

\multicolumn{6}{l}{\textbf{Constraints}} \\
Dense  & 0 & 0 & 0 & - & - \\
$\mathrm{SL}_{\mathrm{chord}}$  & 152 & 310 & 468 & 626 & 784 \\
$\mathrm{SL}_{\mathrm{push}}$  & - & - & - & - & - \\
\bottomrule
\end{tabular}
\caption{Comparison of Dense, $\mathrm{SL}_{\mathrm{chord}}$ and $\mathrm{SL}_{\mathrm{push}}$ for random tensor train polynomials of ranks 2 and local degree 4.}
\label{tab:numerics_r2_d4}
\end{table}

We also study examples with $r=4,d=2$ and relaxation orders to be $k_{\text{dense}}=2n$, $k_{\text{chord}}=3$ and $k_{\mu}=3$. The results in Table \ref{tab:numerics_r4_d2} show how $\mathrm{SL}_{\mathrm{chord}}$ struggles more than $\mathrm{SL}_{\mathrm{push}}$ to solve the SDP, due to now having the rank as the main complexity driver. Note also the impact of using a very low relaxation order for $\mathrm{SL}_{\mathrm{push}}$, which has optimum values that lie lower than the lower bounds given by the dense hierarchy.

\begin{table}[h]
\centering
\begin{tabular}{l@{\hspace{2.5em}}ccccc}
\toprule
 & \multicolumn{5}{c}{$n$} \\
\cmidrule(lr){2-6}
 & 2 & 3 & 4 & 5 & 6 \\
\midrule

\multicolumn{6}{l}{\textbf{Optimum}} \\
Dense  & 0.124422 & -2.125133 & -1.968884 & -3.856409 & - \\
$\mathrm{SL}_{\mathrm{chord}}$  & 0.124422 & - & - & - & - \\
$\mathrm{SL}_{\mathrm{push}}$  & 0.034106 & -2.125132 & -1.968882 & -4.5439500 & -5.086583 \\
\addlinespace

\multicolumn{6}{l}{\textbf{Time (s)}} \\
Dense  & 0.004 & 0.017 & 0.340 & 8.336 & - \\
$\mathrm{SL}_{\mathrm{chord}}$  & 4.844 & - & - & - & - \\
$\mathrm{SL}_{\mathrm{push}}$  & 0.891 & 2.266 & 2.657 & 4.266 & 6.985 \\
\addlinespace

\multicolumn{6}{l}{\textbf{Block size}} \\
Dense  & 3 & 11 & 34 & 94 & 259 \\
$\mathrm{SL}_{\mathrm{chord}}$  & 84 & - & - & - & - \\
$\mathrm{SL}_{\mathrm{push}}$  & 56 & 56 & 56 & 56 & 56 \\
\addlinespace

\multicolumn{6}{l}{\textbf{Constraints}} \\
$\mathrm{SL}_{\mathrm{chord}}$  & 800 & - & - & - & - \\
$\mathrm{SL}_{\mathrm{push}}$  & 17 & 33 & 49 & 65 & 81 \\
\bottomrule
\end{tabular}
\caption{Comparison of Dense, $\mathrm{SL}_{\mathrm{chord}}$ and $\mathrm{SL}_{\mathrm{push}}$ for random tensor train polynomials of ranks 4 and local degree 2.}
\label{tab:numerics_r4_d2}
\end{table}

All of these results reflect the complexity analysis we have made previously, but now we want to see the computational limits of each hierarchy. Note that for some polynomials we may devise a specific indexing for the moment matrices such that we do not take additional unnecessary monomials into account. An example is with tensor train polynomials, where the state variables enter linearly in the problem, then we don't need the monomials that are nonlinear in the state in the indexing. This could reduce the computational cost of both proposed hierarchies, although the indexing would need to be tailored to the specific problem.

\subsection{Scalability study}

Here we test how scalable both $\mathrm{SL}_{\mathrm{chord}}$ and $\mathrm{SL}_{\mathrm{push}}$ are with an increasing number of variables, and also consider what the computational limits are. For this we design a very specific type of tensor train polynomial that will allow us to minimize the effects of numerical conditioning when we go to very large number of variables and degree. 

Take a polynomial $p(x)=u^T\prod_{i=1}^n P_i(x_i)v$ with ranks $r=2$ and local degree 2. The matrix polynomials $P_i(x_i)$ are constructed as a perturbed identities, as follows
\begin{equation}
    P_i(x_i)=I+\sum_{k=1}^d (\frac{\tau}{n}w_k)B_{i,k}(\frac{x_i+1}{2})^k,
    \label{eq:perturbed}
\end{equation}
with constraints $1-x_i^2\geq 0$. We define the quantities in \ref{eq:perturbed} as follows
\begin{itemize}
    \item $\tau$ controls the perturbation size, we let $\tau=0.1$ here.
    \item $w_k \propto 1/k$ make sure lower degrees are weighted more.
    \item $B_{i,k}$ is a random $r\times r$ matrix with coefficients in $[0,1]$ and normalized.
    \item $(\frac{x_i+1}{2})^k$ ensures appropriate scaling on the box constraint.
\end{itemize}

We fix the orders of the relaxations to be $k_{\text{chord}}=3$ and $k_{\mu}=3$.  Now we run the optimization problems for different number of variables, as shown in Table \ref{tab:numerics_large}, and we see how we manage to find the optimal values for up to $n=1000$, although $\mathrm{SL}_{\mathrm{chord}}$ takes 36 minutes to obtain the last value which is due to the exploding number of constraints of the problem. Note how we expect an optimal value of exactly $2$, following from the identity perturbation structure we have designed, so there is a computational error present in each calculation, going up to $4\%$ and $2\%$ for $n=1000$.

\begin{table}[H]
\centering
\begin{tabular}{l@{\hspace{2em}}ccccc}
\toprule
 & \multicolumn{5}{c}{$n$} \\
\cmidrule(lr){2-6}
 & 10 & 50 & 100 & 200 & 1000 \\
\midrule

\multicolumn{6}{l}{\textbf{Optimum}} \\
$\mathrm{SL}_{\mathrm{chord}}$  & 2.000001 & 2.000001 & 2.001614 & 2.048100 & 2.075775 \\
$\mathrm{SL}_{\mathrm{push}}$  & 2.000000 & 2.000003 & 2.000006 & 2.000290 & 2.045234 \\
\addlinespace

\multicolumn{6}{l}{\textbf{Time (s)}} \\
$\mathrm{SL}_{\mathrm{chord}}$  & 6.478 & 52.05 & 147.7 & 197.1 & 2174 \\
$\mathrm{SL}_{\mathrm{push}}$  & 0.900 & 1.411 & 2.325 & 10.00 & 30.49 \\
\addlinespace

\multicolumn{6}{l}{\textbf{Block size}} \\
$\mathrm{SL}_{\mathrm{chord}}$  & 56 & 56 & 56 & 56 & 56 \\
$\mathrm{SL}_{\mathrm{push}}$  & 20 & 20 & 20 & 20 & 20 \\
\addlinespace

\multicolumn{6}{l}{\textbf{Constraints}} \\
$\mathrm{SL}_{\mathrm{chord}}$  & 1263 & 6903 & 13953 & 28053 & 140853 \\
$\mathrm{SL}_{\mathrm{push}}$  & 64 & 344 & 694 & 1394 & 6994 \\
\bottomrule
\end{tabular}
\caption{Large-scale comparison of $\mathrm{SL}_{\mathrm{chord}}$ and $\mathrm{SL}_{\mathrm{push}}$ hierarchies for tensor train structured polynomials, with ranks 2 and degree 2.}
\label{tab:numerics_large}
\end{table}

\section{Applications}\label{sec:applications}

Let us consider different application examples, where we can see the two hierarchies at use. We start focusing on the linear case which appears in many problems, and we use each hierarchy for their main strength: $\mathrm{SL}_{\mathrm{chord}}$ for lower rank but higher degree polynomials, and $\mathrm{SL}_{\mathrm{push}}$ for higher rank but lower degree polynomials. Then we consider an example with a nonlinear composition map: we consider a feedforward neural network with polynomial activation functions and study its reachability properties.

\subsection{Markov chain}
In this first application we will consider the example where the objective polynomial has the form of a tensor train. We will work with local polynomials of up to degree 4, so here we will use $\mathrm{SL}_{\mathrm{chord}}$, which can handle large degrees better.

We consider a two-state Markov chain \cite{levin2009markov} where the state distribution at time $i$ is the vector $v_i=\begin{pmatrix}
    \Pr(X_i=0) & \Pr(X_i=1)
\end{pmatrix}$, and initial condition $v_0=\begin{pmatrix}
    1 & 0
\end{pmatrix}$. We interpret the state being $0$ as a working state and the $1$ as a "not working" state. The initial condition sets the state as working for the initial time, and now we want to maximize the probability of the state working at time $n$. For this, we build the transition probability matrices
\[
P_i(x_i)=\begin{pmatrix}
    a(x_i) & 1-a(x_i) \\ b(x_i) & 1-b(x_i)
\end{pmatrix}
\]
under a control variable $x_i\in[-1,1]$, where $a(x_i)$ is the probability of a state remaining working, and $b(x_i)$ is the probability of the state recovering from a not working state to working. 

The probability distribution evolves as $v_i=v_{i-1}P_i(x_i)$ such that 
\[
v_n=v_0 P_1(x_1)P_2(x_2)\dots P_n(x_n).
\]

Therefore, the probability of the state working at time $n$ is
\[
p_n(x)=v_0\big(\prod^n_{i=1} P_i(x_i)\big)v_0^T.
\]

If we use transition probabilities that are polynomial, then we have an $\mathrm{SL}_{\mathrm{chord}}$ formulation that we can optimize. Consequently, let's take 
\[
a(x)=a_0+a_2x^2 \quad b(x)=b_0+b_2x^2.
\]

In order to be able to benchmark $\mathrm{SL}_{\mathrm{chord}}$ to the actual solution, we build an example that allows us to calculate the solution analytically. We take $a(x)=0.95-0.20x^2$ and $b(x)=0.05-0.05x^2$. With these, we know the global optima to maximize the final probability will be when the controls are "turned off" $x_i=0$ for $i=1,\dots,n$. The optimal transition matrices will then be 
\[
P(0)=\begin{pmatrix}
    a_0 & 1-a_0\\
    b_0 & 1-b_0
\end{pmatrix},
\]

which means that the state at time $i$ is $v_i=\begin{pmatrix} p_i & 1-p_i \end{pmatrix}$, such that the probability for the state to be working at time $i+1$ is
\[
p_{i+1}= a_0 p_i + b_0(1-p_i)=b_0+(a_0-b_0)p_i.
\]

To find the probability to which the system will converge we find 
\[
p_{\infty}=b_0+(a_0-b_0)p_{\infty} \Rightarrow p_{\infty}=\frac{b_0}{1-a_0+b_0}.
\]

Putting these together, we find the probability at time $n$ to be
\begin{equation}
    p_n=p_{\infty}+(1-p_{\infty})(a_0-b_0)^n = \frac{1}{2}+\frac{1}{2}(0.9)^n,
\end{equation}

so we expect the probability of remaining in the working state to converge to $0.5$. We illustrate the results obtained via $\mathrm{SL}_{\mathrm{chord}}$ in Figure ~\ref{fig:Markov} (left pane), and we see how its curve behaves exactly as we expect as $n$ increases, which also makes it hard to visualize.
\begin{figure}[H]
    \centering
    \includegraphics[width=0.48\linewidth]{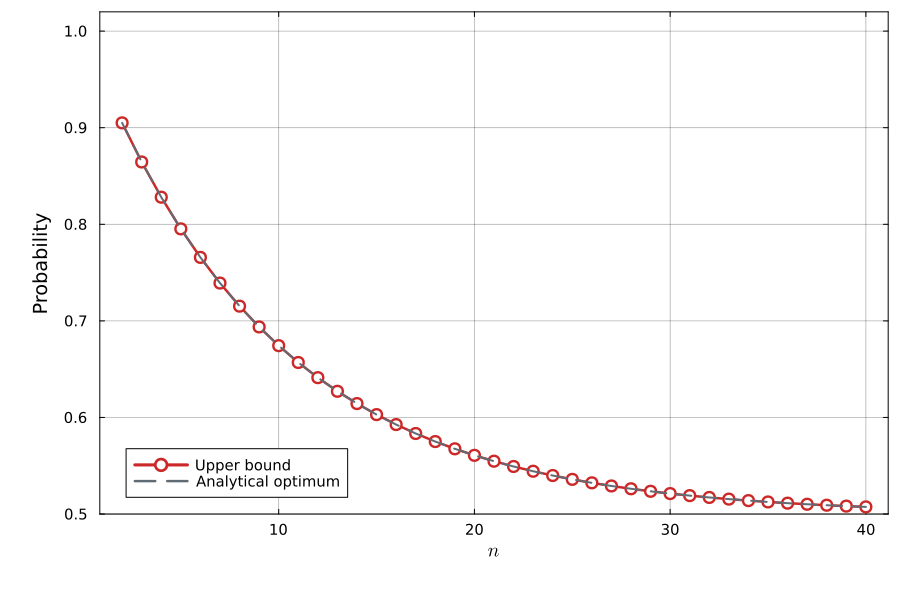}
    \hfill
    \includegraphics[width=0.48\linewidth]{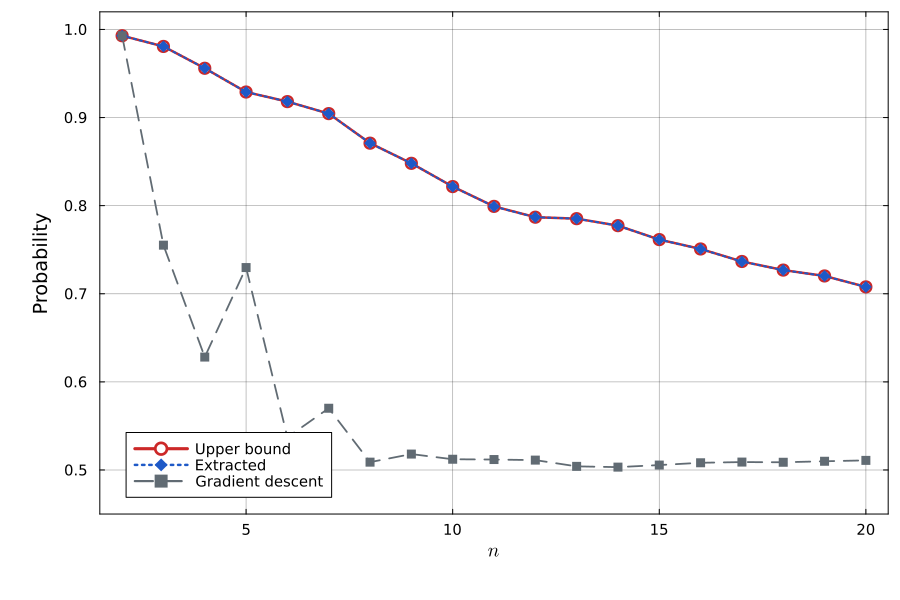}
    \caption{Maximum probability of a 2-state Markov chain with quadratic polynomials as probability transition matrices (left) and with degree-4 frustrated Chebyshev polynomials (right). Comparison between $\mathrm{SL}_{\mathrm{chord}}$ certified upper bound, analytic results, and a local method.}
    \label{fig:Markov}
\end{figure}

Now we consider a Markov chain with changing transition probability matrices for each time step. The functions will now be Chebyshev polynomials of degree 4, which will generate an optimization landscape with more local optima. In this case, we cannot compute an analytical solution so we use projected gradient descent in order to study how $\mathrm{SL}_{\mathrm{chord}}$ does.

In Figure \ref{fig:Markov} (right pane) we compare the upper bound maximum probability obtained by $\mathrm{SL}_{\mathrm{chord}}$ and a projected gradient descent method. We note how $\mathrm{SL}_{\mathrm{chord}}$ provides an upper bound, and projected gradient descent cannot obtain an optimal value close to the $\mathrm{SL}_{\mathrm{chord}}$ value. We also compare these values with the one we obtain by extracting the minimizers from the $\mathrm{SL}_{\mathrm{chord}}$ method and computing the probability. This is done by extracting the first order moments, and it shows how the upper bound is tight with a feasible set of controls that perform better than projected gradient descent.

\subsection{Quantum optimal control}

Here we study the quantum control of a single-qubit \cite{Glaser_2015}. The use of polynomial optimization for quantum optimal control has already been studied in \cite{qcpop,ugs}, but here we consider the sequential application of rotations to the qubit, which allows for the structure to be exploited with our new methods.

The state is a unit vector $s_k \in \R^3, \|s_k\|_2^2=1$, and for each control step we apply a rotation about the $y$ or $z$ axis. To parametrize the problem in a natural way we let 
\[
x_k=\sin(\theta_k), \quad y_k=\cos(\theta_k),\quad x_k^2+y_k^2=1,
\]
and we constrain the time/energy used at each step by setting $|\theta_k|\leq \theta_{\text{max}}\leq \pi/2$, and equivalently $y_k=\cos(\theta_k)\geq \cos(\theta_{\text{max}})=c_{\text{min}}$, which we implement as a localizing constraint. We say the constraints are on time/energy because the control parameters $\theta_k$ come from the time-step size of the rotation matrix, which we can also be defined in terms of energy.

Note that in this problem we will have two variables in each matrix, which means that the maximal clique size will increase by one, as in each clique will include the extra variable we use for the parametrization. This will not lead to numerical problems as we use the $\mathrm{SL}_{\mathrm{push}}$ hierarchy in this example, which allows for higher ranks or number of variables to a lower computational cost compared to $\mathrm{SL}_{\mathrm{chord}}$.

For each step of the time evolution of the system, $k$, we define the state update as 
\begin{equation}
    s_k=P_k(\theta_k)s_{k-1},
\end{equation}
with, letting $(x_k,y_k)=(\sin(\theta_k),\cos(\theta_k))$,
\begin{itemize}
    \item $\text{odd }k: P_k(x_k,y_k)= P_y(x_k,y_k)=\begin{pmatrix}
        y_k & 0 & x_k \\ 0 & 1& 0 \\
        -x_k &0 & y_k
    \end{pmatrix}$,
    \item $\text{even } k:P_k(x_k,y_k)= P_z(x_k,y_k)=\begin{pmatrix}
        y_k & - x_k & 0 \\ 
        x_k & y_k & 0 \\ 0 & 0 & 1
    \end{pmatrix}$.
\end{itemize}

We now aim to solve the quantum optimal control problem of finding the control parameters $\theta_k$ that drive a given initial state to a target state. For our example we let $s_0=(0,0,1)^T$ and $s_{\text{target}}=(0,1,0)^T$, with a time/energy constraint of $\theta_{\text{max}}=0.1$. The POP is then
\begin{align}
    \max_{x_k,y_k,s_k} &\quad s_n \cdot s_{\text{target}}\\
    \text{s.t.} &\quad s_k=P_k(x_k,y_k)s_{k-1},\quad \|s_k\|_2^2=1,\quad x_k^2+y_k^2=1,\quad y_k\geq c_{\text{min}}.
\end{align}

The first case we study is a so-called time optimal control problem, so we want to know how many gates, or steps, we need to implement in order to reach the desired target state to a set accuracy. For this we solve the POP for different number of steps, $n=N$, and plot the optimal value in Figure \ref{fig:toc}. We see how under our constraints, we expect the optimal time to be no less than for $N=43$ rotations, so we get a lower bound for the speed of the system to reach the target state.

\begin{figure}[H]
    \centering
    \includegraphics[width=0.75\linewidth]{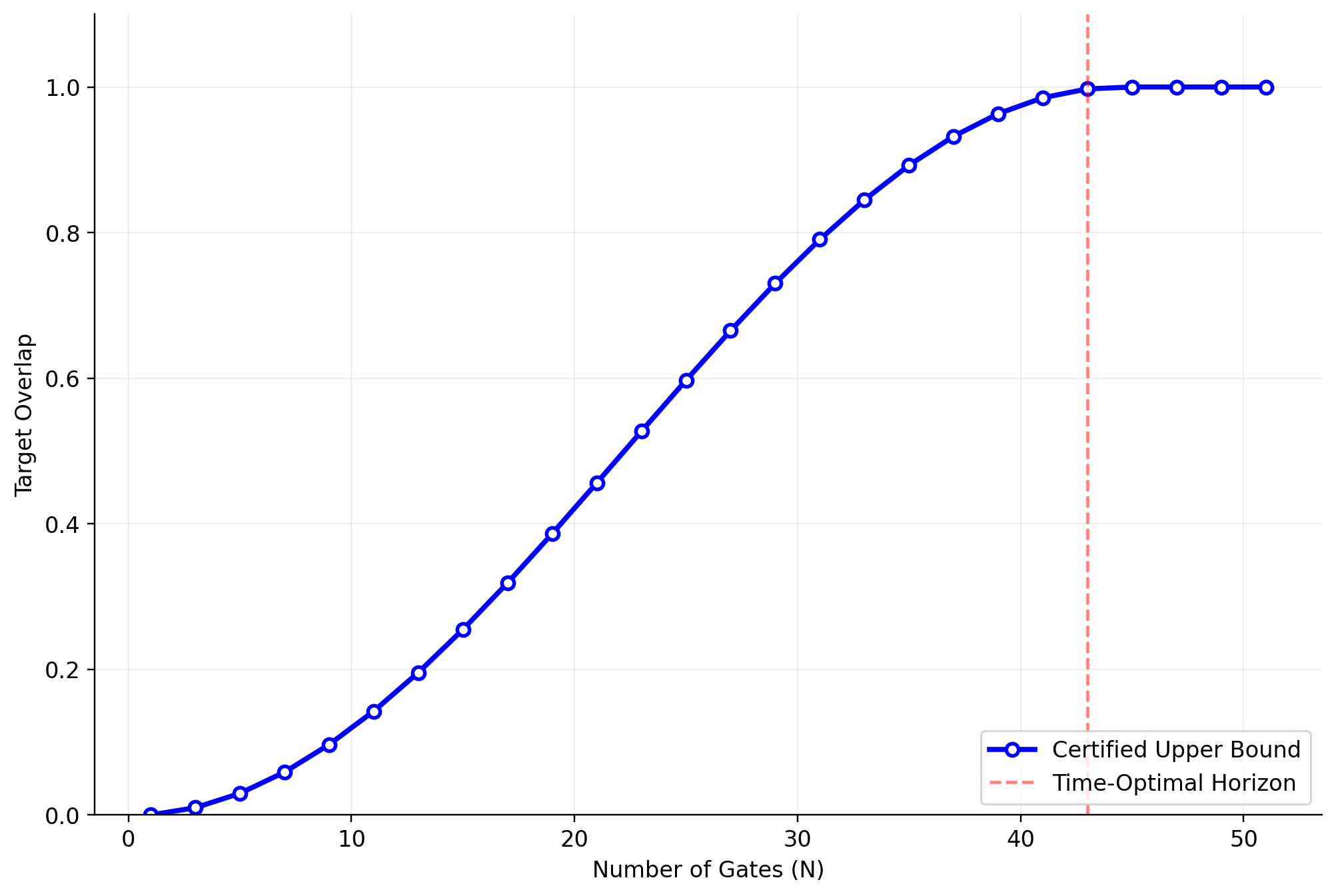}
    \caption{Maximum overlap achieved for different number of steps (gates) using $\mathrm{SL}_{\mathrm{push}}$, and yielding upper bounds on the overlap, which therefore result in a lower bound for the optimal time.}
    \label{fig:toc}
\end{figure}

We can also consider the inverse problem, for a fixed number of gates, determine a control sequence that steers a given initial state as close as possible to a target state. For this, we study the unconstrained problem with a set number of time steps (gates) of $n=5$, initial state $s_0=(0,0,1)^T$ and target state $s_{\text{target}}=(1,0,0)^T$. We use $\mathrm{SL}_{\mathrm{push}}$ to find the maximal overlap between the evolved state and the target state and we obtain $0.999999996$. We want to know which control sequence yields this optimal value, so we use a heuristic extraction method described in detail below in order to get $\theta_{\text{opt}}=\{-0.910, -3.139, 0.675, 0.003, -0.014\}$. If we plug these optimal controls back into the dynamics to compute the final state we get an overlap of $0.999999489$ with the target, almost as good as the optimal value.

\subsubsection{Sequential randomized extraction}


In the quantum control example, the push-forward hierarchy does not give one single global moment sequence for the optimal trajectory. Instead, it gives local moment data associated with each stage of the trajectory, and this changes the computational form of the extraction problem.

The motivation for our extraction procedure is not merely that the push-forward hierarchy involves several measures, since this also happens in sparse moment-SOS hierarchies. The key point is that these measures are different time steps and coupled through the state recursion, such that the object we want to recover is a single feasible trajectory. Standard extraction methods are designed to recover atoms from moment data once a suitable flatness structure is present, but here a stagewise extraction might not produce controls that remain compatible from one time step to the next. An alternative is suggested by Proposition \ref{prop:gluing}, one could first extract atoms from the local stage measures and then glue them through their common state marginals to obtain a measure on the full trajectory space. We do not pursue this approach here since it requires the flatness condition to hold. Instead, we use a sequential local heuristic that constructs one trajectory step by step while enforcing compatibility with the previously selected state.

At each step $k$, the push-forward hierarchy returns a truncated moment sequence for a transition measure $\mu_k$ over the variables $z_k=(s_{k-1},x_k,y_k)$. If $\mu_k$ is nearly an atomic measure then the first moments can already provide a good approximation for the point $(s_{k-1},x_k,y_k)$. However, when $\mu_k$ is not close to atomic, the first moments might simply average several different candidate controls, which might not correspond to a feasible choice. In this case, we use the local moment matrix $M_k$ associated to $\mu_k$ to generate candidate points.

Let \(M_k\) be the moment matrix obtained from the SDP solution at step \(k\), indexed by a monomial basis
\[
v_k(z_k)=\bigl(1,\ s_{k-1,1},\dots,s_{k-1,r_{k-1}},\ x_k,\ y_k,\ \text{higher-degree monomials in } z_k\bigr),
\]
where \(z_k=(s_{k-1},x_k,y_k)\). Compute a factorization $M_k = V_k V_k^\top$, for example from an eigenvalue decomposition. We then sample a Gaussian vector \(g\sim \mathcal N(0,I)\) and define $\omega = V_k g$. The random vector \(\omega\) is indexed by the same monomial basis \(v_k(z_k)\), and has covariance matrix \(M_k\). In this way, the sample reflects the second-order moment information encoded by \(M_k\).

To obtain a candidate point in the original variables, we use the entries of \(\omega\) corresponding to the constant and degree-one monomials. Denoting by \(\omega_1\) the entry associated with the constant monomial \(1\), by \(\omega_{s_{k-1,i}}\) the entry associated with \(s_{k-1,i}\), and by \(\omega_{x_k},\omega_{y_k}\) the entries associated with \(x_k,y_k\), we define
\[
\widetilde s_{k-1,i}=\frac{\omega_{s_{k-1,i}}}{\omega_1},\qquad
\widetilde x_k=\frac{\omega_{x_k}}{\omega_1},\qquad
\widetilde y_k=\frac{\omega_{y_k}}{\omega_1}.
\]
If $|\omega_1|$ is below a prescribed threshold, we discard the sample and resample.

Thus we interpret the sampled linear monomial entries of \(\omega\), normalized by the constant entry, as a candidate value
\[
\widetilde z_k=(\widetilde s_{k-1},\widetilde x_k,\widetilde y_k)
\]
for the state-control variables at step \(k\).

Since we are looking for a single trajectory that is consistent across all stages, the candidate previous state extracted from the local sample at step $k$ must be compatible with the trajectory already constructed from the earlier steps. To make this precise, we define the propagated state recursively by
\[
s_0^{\mathrm{pf}} := s_0,
\qquad
s_j^{\mathrm{pf}} := P_j(\theta_j^{\mathrm{pf}})\, s_{j-1}^{\mathrm{pf}},
\quad j=1,\dots,k-1,
\]
where \(\theta_1^{\mathrm{pf}},\dots,\theta_{k-1}^{\mathrm{pf}}\) are the controls selected in the previous extraction steps. Thus \(s_{k-1}^{\mathrm{pf}}\) is the state reached after propagating the already extracted controls up to stage \(k-1\).

We accept the candidate $\widetilde{z}_k$ only if its previous-state component is sufficiently close to the propagated state, namely if
\[
\|\widetilde s_{k-1}-s_{k-1}^{\mathrm{pf}}\|\le \tau,
\]
for some tolerance parameter $\tau>0$. This enforces compatibility between the local candidate extracted from the moment data at step $k$ and the trajectory constructed from the previously selected controls.

For each accepted candidate, we next enforce feasibility of the control variables with respect to the constraints of the quantum control problem. In our setting, the variables must satisfy $x_k^2+y_k^2=1$, so we project $(\widetilde x_k,\widetilde y_k)$ onto the unit circle by
\[
(\widetilde x_k,\widetilde y_k)\leftarrow
\frac{(\widetilde x_k,\widetilde y_k)}
{\sqrt{\widetilde x_k^2+\widetilde y_k^2}}.
\]
We then recover the corresponding control angle by
\[
\widetilde\theta_k=\operatorname{atan2}(\widetilde x_k,\widetilde y_k).
\]
If additional bounds on the angle are imposed, we enforce them at this stage as well, by projecting $\widetilde\theta_k$ to the admissible interval and redefine $(\widetilde x_k,\widetilde y_k)=(\sin \widetilde\theta_k, \cos \widetilde\theta_k)$. 

To score the accepted candidates, we propagate the previously selected state \(s_{k-1}^{\mathrm{pf}}\) using each candidate angle \(\widetilde\theta_k\), namely
\[
\widetilde s_k = P_k(\widetilde\theta_k)\, s_{k-1}^{\mathrm{pf}}.
\]
Among all accepted candidates, we select the one maximizing the overlap $\widetilde s_k\cdot s_{\mathrm{target}}$. Denoting the selected candidate by \(\theta_k^{\mathrm{pf}}\), we then update the propagated state by
\[
s_k^{\mathrm{pf}} = P_k(\theta_k^{\mathrm{pf}})\, s_{k-1}^{\mathrm{pf}}.
\]
Repeating this procedure for \(k=1,\dots,n\) yields a deterministic sequence of extracted control angles $\theta_1^{\mathrm{pf}},\dots,\theta_n^{\mathrm{pf}}$, and the associated propagated trajectory $s_1^{\mathrm{pf}},\dots,s_n^{\mathrm{pf}}$.

This is a heuristic rounding procedure and is not guaranteed to recover a globally optimal trajectory. In practice, we therefore repeat it for several random seeds and retain the best trajectory obtained. The procedure is summarized in Algorithm \ref{alg}. If $\mathcal{C}_k=\emptyset$ after processing the $N_{\mathrm{samp}}$ samples, we continue sampling until one candidate is accepted.

\begin{algorithm}[H]
\caption{Sequential randomized extraction for push-forward hierarchy}\label{alg}
\begin{algorithmic}[1]
\State \textbf{Input:} moment matrices \(M_k\), initial state \(s_0\), target state \(s_{\mathrm{target}}\), tolerance \(\tau\), number of samples \(N_{\mathrm{samp}}\)
\State Set \(s_0^{\mathrm{pf}}=s_0\)
\For{\(k=1,\dots,n\)}
    \State Compute a factorization \(M_k=V_kV_k^\top\)
    \State Initialize the set of admissible candidates \(\mathcal C_k=\emptyset\)
    \For{\(m=1,\dots,N_{\mathrm{samp}}\)}
        \State Sample \(g^{(m)}\sim\mathcal N(0,I)\) and set \(\omega^{(m)}=V_kg^{(m)}\)
        \State Extract \((\widetilde s_{k-1}^{(m)},\widetilde x_k^{(m)},\widetilde y_k^{(m)})\) from the linear entries of \(\omega^{(m)}\)
        \If{\(\|\widetilde s_{k-1}^{(m)}-s_{k-1}^{\mathrm{pf}}\|\le \tau\)}
            \State Project \((\widetilde x_k^{(m)},\widetilde y_k^{(m)})\) onto the feasible set
            \State Compute the corresponding angle \(\widetilde\theta_k^{(m)}\)
            \State Propagate
            \[
            \widetilde s_k^{(m)}=P_k(\widetilde\theta_k^{(m)})\,s_{k-1}^{\mathrm{pf}}
            \]
            \State Add \((\widetilde\theta_k^{(m)},\widetilde s_k^{(m)})\) to \(\mathcal C_k\)
        \EndIf
    \EndFor
    \State Select
    \[
    (\theta_k^{\mathrm{pf}},s_k^{\mathrm{pf}})
    \in
    \arg\max_{(\theta,s)\in\mathcal C_k} s\cdot s_{\mathrm{target}}
    \]
\EndFor
\State \textbf{Output:} extracted control sequence \((\theta_1^{\mathrm{pf}},\dots,\theta_n^{\mathrm{pf}})\)
\end{algorithmic}
\end{algorithm}

\subsection{Neural network}

We consider a feedforward neural network (FNN) with polynomial activations \cite{goodfellow2016deep} to see a possible application of compositional polynomials. Polynomial optimization has already been used in the context of neural networks; for example \cite{Milan_NN} studies stability and performance, and \cite{NEURIPS2020_dea9ddb2} studies the robustness of neural networks  by computing Lipschitz constants using polynomial optimization. Here we use it in a context of worst-case testing, for which we compute the upper and lower bounds on the state the neural network can take at each step.

We have decision variables $x_1,\dots,x_n\in[-1,1]$ and states $s_i\in \R^r$; we fix $r=3$ for this example, for each stage $i=1,\dots,n$, and initialized at $s_0=0$. The neural network applies the following maps

\begin{equation}
    u_i=A_is_{i-1}+b_ix_i+c_i, \quad s_i=F(u_i), \quad F(t)=t+\alpha t^3,
\end{equation}
for component-wise cubic $t$. The parameters $A_i,b_i,c_i$ are generated randomly once per stage with scaling to keep them small, as follows: $A_i\sim 0.1\cdot\mathcal{N}(0,1), b_i\sim 0.3\cdot \mathcal{N}(0,1), c_i\sim 0.05\cdot \mathcal{N}(0,1)$. 

The problem of interest is to find certificates for both maximum and minimum values that the state $s_{i,1}$ can attain for stages $i=1,\dots,N$, and we fix $N_{\max}=20$, we use $\mathrm{SL}_{\mathrm{chord}}$ for the POPs due to the high degree nature of the polynomial. With these we can find an envelope that tells us the region where all possible state values can lie in. We test this in Figure \ref{fig:tube} by comparing the certified envelope with 50 random trajectories, all of which lie within the envelope and furthermore it allows us to check that for this example we find tight bounds.

\begin{figure}[H]
    \centering
    \includegraphics[width=0.75\linewidth]{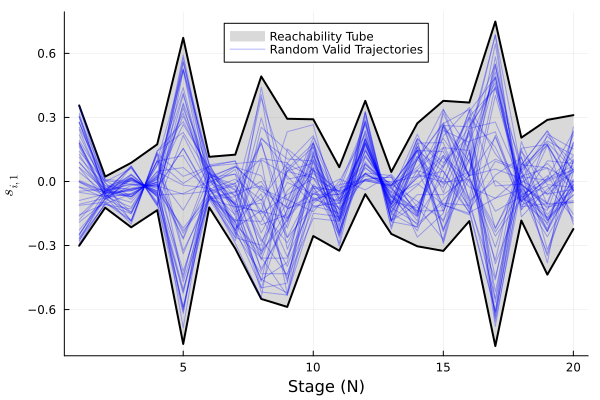}
    \caption{Range of values that the state can take at each stage of the neural network, compared to randomly sampled trajectories.}
    \label{fig:tube}
\end{figure}

This application is a demonstration of a worst-case study, as it  tests the worst-case properties of the system. If we interpret the input data as noise we can also use this example to see how sensitive our system is to disturbance and get a robustness certificate that bounds the trajectories.

\section{Conclusion and perspectives}

This work can be seen as a second step, following the introduction of LRPOP, towards a broader optimization framework for structured high-dimensional polynomials. The chain structure studied here is an example of the fundamental mechanism: writing a highly-complex polynomial as the low-dimensional propagation of intermediate states. This structure can be exploited through sparsity at a graph level or through the underlying measure propagation. From this perspective, chain composition and tensor trains are only a particular example of a more general framework.

A natural next step is then to move beyond chain structures and consider other tensor network topologies. In this setting, the notion of complexity will not only be given by the intermediate state dimensions, but also by the geometry of the tensor graph and the way information is transmitted within it. We expect that this could lead to new moment-SOS hierarchies for trees, and even more general acyclic networks.

Another promising direction is combining these ideas with data modeling. In many applications, the structured representation of the high-dimensional function is not given a priori, but inferred from data. Using both tensor network approximation or model learning with certified polynomial optimization, first learning the low-dimensional approximate model and then optimizing it with guarantees, could provide new links with areas like system identification or reduced-order modeling, where rigorous certificates of optimality are difficult to obtain.

We hope that the ideas developed here help connect two communities that often follow different paths like tensor decomposition for model approximation and scalability, and moment-SOS methods for certification and global guarantees in optimization problems. Bridging these viewpoints may open new lines of work for optimization methods where we not only have scalability but also mathematical certifications.

\section*{Acknowledgements}

We would like to thank Charles Poussot-Vassal, Matteo Rizzi and Henrique Goulart for insightful discussions on tensor decomposition and tensor networks. The authors acknowledge the use of AI for assistance with brainstorming ideas, mathematical development, coding and drafting the manuscript. The final content, analysis and conclusions remain the sole responsibility of the authors. 

This work has been supported by European Union’s HORIZON–MSCA-2023-DN-JD programme under the Horizon Europe (HORIZON) Marie Skłodowska-Curie Actions, grant agreement 101120296 (TENORS). This work was also funded by the European Union under the project ROBOPROX (reg.~no.~CZ.02.01.01/00/22\_008/0004590).

\end{document}